\documentclass[10pt]{article}

\providecommand{\bbP}{\mathbb{P}}

\providecommand{\CF}{\mathcal{F}}

\usepackage{authblk}

\usepackage{hyperref}
\usepackage{pdfpages}
\usepackage[toc,page]{appendix}
\usepackage{upgreek}
\usepackage{fullpage} 
\usepackage[english]{babel} \usepackage[latin1]{inputenc}
\usepackage{amsmath} \usepackage{xcolor} \usepackage{graphicx}
\usepackage{amsthm} 
\usepackage{psfrag}
\usepackage{subfig} \usepackage{amssymb} \usepackage{tikz} 
\usepackage{ifthen} \usepackage{xifthen} \usepackage{dsfont}
\usepackage{bm} \usepackage{accents}
\usepackage{xparse} %
\usepackage{mathtools} %
\usetikzlibrary{arrows,calc} \usetikzlibrary{decorations,arrows}
\usetikzlibrary{decorations.pathreplacing}

\newcommand{\todo}[1]{\ifthenelse{\isempty{#1}}{\textcolor{red}{\textbf{TODO}}}{\textcolor{red}{\textbf{[TODO: #1]}}}}
\newcommand{\temp}[1]{} \newcommand{\norm}[2][]{\| #2 \|_{#1}}
\newcommand{\hide}[1]{} 

\newcommand{\normc}[2][]{\left\| #2 \right\|_{#1}}
\newcommand{\setc}[2]{\left\{#1\, :\,#2 \right\}}
\newcommand{\set}[2]{\{#1\,:\,#2\}}

 \newcommand{\dd}{\mathrm{d}}
 \DeclareMathOperator{\supp}{supp}

\DeclareMathOperator*{\essinf}{ess\,inf}

\DeclareMathOperator{\e}{e}

\DeclareMathOperator{\size}{size}
\DeclareMathOperator{\depth}{depth}
\newcommand{\ii}{\mathrm{i}}

\newcommand{\dom}{\mathrm{D}}

\newcommand{\cB} {{\mathcal B}}

\newcommand{\cS} {{\mathcal S}}
\newtheorem{theorem}{Theorem}[section]

\newtheorem{corollary}[theorem]{Corollary}

\newtheorem{definition}[theorem]{Definition}

\newtheorem{lemma}[theorem]{Lemma}

\newtheorem{proposition}[theorem]{Proposition}
\newtheorem{remark}[theorem]{Remark}

\newtheorem{assumption}[theorem]{Assumption}

\numberwithin{equation}{section}

\newcommand{\cF}{{\mathcal F}}

\newcommand{\C}{{\mathbb C}}
\newcommand{\E}{{\mathbb E}}
\newcommand{\N}{{\mathbb N}}
\newcommand{\R}{{\mathbb R}}
\newcommand{\Z}{{\mathbb Z}}

\newcommand{\eps}{\varepsilon}
\newcommand{\bsb}{{\boldsymbol b}}

\newcommand{\bsy}{{\boldsymbol y}}
\newcommand{\bstau}{{\boldsymbol{\tau}}}
\newcommand{\bsbeta}{{\boldsymbol{\beta}}}
\newcommand{\bsx}{{\boldsymbol{x}}}

\newcommand{\bsz}{{\boldsymbol z}}

\newcommand{\bsvarrho}{{\bm \varrho}}
\newcommand{\bsnu}{{\bm \nu}}
\newcommand{\bse}{{\bm e}}
\newcommand{\bsmu}{{\bm \mu}}

\newcommand{\bsnul}{{\bm 0}}

\newcommand{\dup}[2]{\langle #1, #2\rangle}

\newcommand{\ReLU}{{\mathrm{ReLU}}}

\begin{document}
\date{\today} \title{Deep Learning in High Dimension: Neural Network
  Approximation of Analytic Functions in $L^2(\R^d,\gamma_d)$}

\author[1]{Christoph Schwab}
\author[2]{Jakob
Zech%
} \affil[1]{\footnotesize Seminar for Applied Mathematics, ETH
  Z\"urich, 8092 Z\"urich, Switzerland,
  \href{mailto:schwab@math.ethz.ch}{schwab@math.ethz.ch}}
\affil[2]{\footnotesize Interdisziplin\"ares Zentrum f\"ur
  wissenschaftliches Rechnen, Universit\"at Heidelberg, 69120
  Heidelberg, Germany,
  \href{mailto:jakob.zech@uni-heidelberg.de}{jakob.zech@uni-heidelberg.de}}

\maketitle

\abstract{For artificial deep neural networks, we prove expression
  rates for analytic functions $f:\R^d\to\R$ in the norm of
  $L^2(\R^d,\gamma_d)$ where $d\in \N\cup\{ \infty \}$.  Here
  $\gamma_d$ denotes the Gaussian product probability measure on
  $\R^d$.  We consider in particular $\ReLU$ and $\ReLU^k$ activations
  for integer $k\geq 2$.  For $d\in\mathbb{N}$, we show exponential
  convergence rates in $L^2(\mathbb{R}^d,\gamma_d)$.  In case
  $d=\infty$, under suitable smoothness and sparsity assumptions on
  $f:\R^\N\to\R$, with $\gamma_\infty$ denoting an infinite (Gaussian)
  product measure on $(\R^\N, \cB(\R^\N))$, we prove
  dimension-independent expression rate bounds in the norm of
  $L^2(\R^\N,\gamma_\infty)$.  The rates only depend on quantified
  holomorphy of (an analytic continuation of) the map $f$ to a product
  of strips in $\C^d$.  %
  As an application, we prove expression rate bounds of deep
  $\ReLU$-NNs for response surfaces of elliptic PDEs with log-Gaussian
  random field inputs.}
  \section{Introduction}
  \label{sec:Intro}
  This paper addresses the approximation of analytic functions
  $f:\R^d\to\R$ by deep neural networks (DNNs for short) in the space
  $L^2(\R^d,\gamma_d)$.  Here $\gamma_d$ denotes the $d$-fold product
  Gaussian measure, with $d\in \mathbb{N}\cup\{\infty\}$.  To quantify
  DNN expression rates, we assume $f$ to belong to a class of
  functions that allows holomorphic extensions to certain cartesian
  products of strips around the real line in the complex plane. This
  implies summability results on coefficients in Wiener-Hermite
  polynomial chaos expansions of $f$.  We separately discuss the
  finite dimensional case $d\in\N$ and the (countably) infinite
  dimensional case $d=\infty$.  Our expression rate analysis is based
  on expressing such functions through their finite- or
  infinite-parametric Wiener-Hermite polynomial chaos (gpc) expansion.
  Reapproximating the gpc expansion, we provide DNN architectures and
  corresponding DNN size bounds which show that such functions can be
  approximated at an exponential convergence rate in finite dimension
  $d\in\N$.  For $d=\infty$, i.e. in the infinite dimensional case,
  our DNN expression rate bounds are free from the so-called curse of
  dimensionality: we prove that in this case our DNN expression rate
  bounds are only determined by the summability of the gpc expansion
  coefficient sequences.  Thus, while we concentrate on analytic
  functions, the scope of our results extends to statistical learning
  of any object that can be represented as a Wiener-Hermite expansion
  with bound on summability of the coefficient sequences.

  Relevance of the present investigation derives from the fact that
  functions belonging to the above described class arise in particular
  as response maps in uncertainty quantification (UQ) for partial
  differential equations (PDEs for short) with Gaussian random field
  inputs. Modelling unknown inputs of elliptic or parabolic PDEs by a
  log-Gaussian random field, the corresponding PDE response surface
  can under certain assumptions be shown to be of this type
  \cite{DNSZ20}. We discuss a standard example in
  Sec.~\ref{sec:PDESol} ahead.  As such, our results have broad
  implications for a wide range of problems in forward and inverse UQ.
  Dating back to the seminal works \cite{NWienHomogPC,CamMart47} the
  numerical approximation of Gaussian Random Fields (GRFs for short)
  and response maps with GRF inputs by truncated Hermite polynomial
  chaos expansions has received substantial attention during recent
  years, specifically due to the ubiquitous role of GRFs in spatial
  statistics, theoretical physics, data assimilation, and stochastic
  Partial Differential Equations (PDEs for short).  We refer to the
  surveys \cite{bogachGM,LifshitsBook,Janson}, to the recent
  publications \cite{StTeckGPAppr2018,LS15_1111} 
  and to
  the references there for the discussion of GRFs, as well as to,
  e.g., \cite{HS,BCDM} and the references there for the approximation
  of PDE response surfaces with log-GRF inputs.
\subsection{Previous results}
\label{sec:PrevReS}
In recent years, there has been substantial activity in the
  analysis of expression rates of $\ReLU$-DNNs for various classes of
  functions.  We mention for instance the papers
  \cite{YAROTSKY2017103,pmlr-v75-yarotsky18a} which established
  optimal convergence rates for functions of finite regularity.
  Approximation in $L^p$-spaces was discussed in
  \cite{PETERSEN2018296}. In \cite{OPS19}, DNN expression rates were
  given for functions from Sobolev- and Besov-spaces, as well as for
  certain classes of analytic functions. Holomorphic functions of
  $d\in \N$ many variables on bounded domains were shown to admit
  exponential expression rates by deep $\ReLU$-NNs in
  \cite{OSZ19}. The case of infinite-parametric holomorphic functions
  on cartesian products of bounded intervals was discussed in
  \cite{SZ18}. The analysis there is conceptually closely related to
  the present work.  In this reference, we proved deep $\ReLU$-NN
  expression rate bounds for gpc representations of
  countably-parametric functions on $[-1,1]^\N$. The obtained
  approximation rates do not suffer from the curse of dimensionality,
  and were shown %
  to be governed only by a suitable notion of sparsity, as quantified
  in terms of summability of gpc coefficients. %
  Importantly, with the exception of \cite{SZ18}, all results in these
  references addressed approximation rate bounds for functions defined
  on bounded subdomains $\dom$ of Euclidean space $\R^n$ with
  moderate, fixed ``physical'' dimension $n\in \N$.  Also in other
  contexts, $\ReLU$-NN expression rate bounds are often stated and
  proved for DNNs with bounded input ranges. On bounded intervals,
  $\ReLU$-NNs afford in particular the efficient emulation of
  orthogonal Jacobi polynomials.

  Our previous paper \cite{SZ18} is conceptually 
  closely related to the present work.
    In this reference, we proved deep $\ReLU$-NN expression rate
    bounds for generalized polynomial chaos (``gpc'' for short)
    representations of countably-parametric functions which
    approximation rates do not suffer from the so-called curse of
    dimensionality.  The DNN expression rates of such functions were
    shown in \cite{SZ18} to be governed only by a suitable notion of
    sparsity, as quantified in terms of summability of gpc
    coefficients.  In \cite{SZ18}, we only considered bounded
    parameter domains, and gpc expansions with respect to polynomials
    that are orthonormal with respect to probability measures on these
    domains. In particular, Legendre and Jacobi polynomials. Although
    the present results are in a similar spirit as the results in
    \cite{SZ18}, they do not follow from these results, but differ
    both in statement and proofs in an essential way from the results
    in \cite{SZ18}.  Similar to \cite{SZ18}, the presently obtained
    expression rate bounds will be based on known (in part rather
    recent) bounds on approximation rates of $n$-term Hermite gpc
    expansions of GRFs, from \cite{DNSZ20}.

    Deep Neural Networks (DNNs) have seen intense research activity,
    mainly driven by successes in practical deep learning approaches
    in the emerging field of data science. This momentum has also
    initiated new developments in the numerical solution of PDEs,
    being based on DNNs as approximation architectures rather than
    ``traditional'' approaches built on Finite Element or Spectral
    methods. In practical applications, at times spectacular
    performance (in terms of accuracy versus DNN size) has been
    reported. These practical findings have been recently supported by
    theory indicating that $\ReLU$ DNNs can, indeed, emulate a wide
    range of linear approximation methods in classical function
    systems such as splines, multiresolution systems, polynomials,
    Fourier series, etc.  Here, $\ReLU$-NNs with suitable
    architectures afford with corresponding expression rate bounds
    which are equal, or only slighly inferior to rates afforded by the
    mentioned systems (see, e.g., \cite{OPS19,OSZ19} and the
    references there).  Importantly, all results in these references
    addressed approximation rate bounds for functions defined on
    bounded subdomains $\dom$ of euclidean space $\R^n$ with moderate,
    fixed ``physical'' dimension $n\in \N$.  Also in other contexts,
    $\ReLU$-NN expression rate bounds are often stated and proved for
    DNNs with bounded input ranges. On bounded intervals, $\ReLU$-NNs
    afford in particular the efficient emulation of orthogonal Jacobi
    polynomials.

    The expression rate analysis of polynomial function systems on
    unbounded domains has received less attention.  In view of the
    wide use of Gaussian process (GP for short) models and of Gaussian
    random fields in statistical modelling of uncertainty, and in
    theoretical physics \cite{Janson}, and due to the close connection
    of Hermite orthogonal polynomials with the Gaussian measure
    (e.g. \cite{szego,NWienHomogPC,LS15_1111} and the references
    there), expression rates of DNNs for Hermite polynomials in mean
    square with respect to Gaussian measure over $\R$ are crucial for
    restablishing various approximation rate bounds for Gaussian
    random fields, and in particular for operator equations with
    Gaussian random field inputs.  The present paper addresses this
    question. The focus is on DNNs with so-called $\ReLU$ activation
    function.
    Despite these specificities of $\ReLU$-NNs, our DNN architectures
    and expression rate bounds, which are explicit in the polynomial
    degree and in the accuracy, are valid also for wider families of
    activation functions. We expect that similar
    arguments allow to prove expression rate bounds also for other
    (smoother) activation functions.
  \subsection{Contributions}
  \label{sec:Contr}
  The present paper has the following principal contributions.
  \begin{enumerate}
  \item We prove expression rate bounds for deep $\ReLU$-NNs of
    univariate ``probabilistic'' Hermite polynomials $H_n$ of
    polynomial degree $n\in\N_0$ on $\R$, in $L^2(\R,\gamma_1)$,
    i.e. in mean square with respect to Gaussian measure $\gamma_1$ on
    $\R$. This result is then generalized to multivariate Hermite
      polynomials, see Thm.~\ref{cor:Mneps} and Thm.~\ref{thm:reluhermite}.

    \item In the case of finite parameter dimension $d$, we establish
      exponential convergence in $L^2(\R^d,\gamma_d)$ for the
      approximation of certain analytic functions by deep $\ReLU$-NNs.
      See Thm.~\ref{thm:finite_relu}.
    
  \item 
    In infinite dimension, for a class of infinite parametric
    functions satisfying an analyticity condition, we prove $\ReLU$-NN
    expression rate bounds $L^2(\R^\N, \gamma_\infty)$ that are free
    from the %
    curse of dimension %
    with explicit account of the NN size and depth. See Thm.~\ref{thm:DNNGRF}.

  \item As an example, we show how our result in infinite dimensions
    implies $\ReLU$-NN expression rate bounds for response surfaces of
    elliptic PDEs with infinite-parametric, log-Gaussian random field
    input. See Prop.~\ref{prop:DNNGRF}.
  \end{enumerate}
  \subsection{Notation}
  \label{sec:Notat}
  Throughout $C>0$ is used to denote a generic constant that may
    change its value even within the same equation.  Moreover,
    $z=x+\ii y\in\C$ indicates that $z\in\C$ and $x=\Re[z]\in\R$ and
    $y=\Im[z]\in\R$.  In particular, $\ii$ shall denote the imaginary
    unit.
  \subsubsection{Gaussian measures}
  \label{sec:GM}
  For finite $d\in \N$, denote by $\gamma_d$ the standard Gaussian
  measure on $\R^d$.  Its density w.r.t.\ the Lebesgue measure on
  $\R^d$ is given by
  \begin{equation*}
    \frac{1}{(2\pi)^{d/2}}\e^{-\frac{\norm[2]{\bsy}^2}{2}}\qquad\forall\,\bsy\in\R^d,
  \end{equation*}
  where $\norm[2]{\cdot}$ is the Euclidean
  norm. %
  Additionally, $\gamma=\bigotimes_{j\in\N} \gamma_1$ denotes the
  infinite product (probability) measure on $\R^\N$.  We refer to
  \cite[Chapter 2]{bogachGM} for details.  We write
  $L^2(\R^d,\gamma_d)$ for the usual $L^2$ space w.r.t.\ the measure
  $\gamma_d$.  For $d=\infty$ we additionally introduce the shorthand
  notation $U:=\R^\N$, indicating a countable cartesian product of
  real lines, the corresponding $L^2$-space is then $L^2(U,\gamma)$.
  Similarly, for a Banach space $V$ and $k>1$, $L^k(U,\gamma;V)$
  is the Bochner space of functions with values in $V$.
  \subsubsection{Multiindices and polynomials}
  \label{sec:MulIndPoly}
  Throughout, $\N = \{ 1,2,...\}$ and $\N_0 = \{ 0,1,2,...\}$.
  Multi-indices in $\N_0^d$ or $\N_0^\N$ shall be denoted by
  $\bsnu$, i.e.\ $\bsnu = (\nu_j)_{j=1}^d$ or
  $\bsnu = (\nu_j)_{j\in\N}$ respectively. 
  The size (or total order) of the multi-index $\bsnu$ is
  $|\bsnu| := \sum_{j\ge 1} \nu_j$.  For $d=\infty$, by
  $\cF = \{ \bsnu\in \N_0^\N: |\bsnu|<\infty \}$ we denote the
  countable subset of $\N_0^\N$ of multi-indices of ``finite
  support'': if $\bsnu\in \cF$, we let $\supp\bsnu:=\set{j}{\nu_j\neq 0}$ and
  $|\bsnu|_0 := \#(\supp\bsnu)$. Comparison of multi-indices is
  component-wise: we write $\bsmu\leq \bsnu$ iff for every $j$ holds
  $\mu_j \leq \nu_j$.  A finite set $\Lambda\subseteq\N_0^d$ or
  $\Lambda\subseteq\cF$ will be called \emph{downward closed}, iff
  $\bsnu\in\Lambda$ implies $\bsmu\in\Lambda$ whenever
  $\bsmu\in\Lambda$.

  With $\bbP_n:={\rm span}\set{x^j}{j\in\{0,\dots,n\}}$ we denote the
  space of all polynomials of %
  degree at most $n$ with real coefficients.  
  In the multivariate case, for a subset
    $\Lambda\subseteq\N_0^d$ with $d\in\N$ or
    $\Lambda\subseteq\cF$, we write
    $\bbP_\Lambda:={\rm span}\set{\prod_{j\in\supp\bsnu}x_j^{\nu_j}}{\bsnu\in\Lambda}$.
  \subsubsection{Neural networks}
  \label{sec:NNs}
We consider feedforward neural networks without skip connections. 
That is, for a given activation function
$\sigma:\R\to\R$, we consider mappings $\Phi:\R^{n_0}\to\R^{n_{L+1}}$ which can
be represented via
\begin{equation}\label{eq:FFWDNN}
    \Phi = A_L\circ \sigma\circ \cdots\circ\sigma\circ A_0
\end{equation}
for certain linear transformations
$A_j:\R^{n_j}\to \R^{n_{j+1}}:x\mapsto W_jx + b_j$.  
Here
  $W_j\in\R^{n_{j+1}\times n_j}$ are the weight matrices and
  $b_j\in\R^{n_{j+1}}$ are the bias vectors, and the application
  of $\sigma$ in \eqref{eq:FFWDNN} is understood componentwise.
  Such a function $\Phi$ will be called a $\sigma$-NN
  of depth $L$ and size
  \begin{equation*}
    \size(\Phi):=|\set{(j,k,l)}{(W_j)_{k,l}\neq 0}\cup \set{(j,k)}{(b_j)_{k}\neq 0}|.
  \end{equation*}
We also use the notation $\depth(\Phi):=L$. Hence the depth
corresponds to the number of applications of the activation
function, and the size corresponds to the number of nonzero weights
and biases in the network.
\subsection{Layout}
\label{sec:Struct}
  The structure of the paper is as follows.  
  In Section~\ref{sec:HermPol}, we recapitulate general definitions and
  classical properties of Hermite polynomials.
  Section~\ref{sec:PrlBnd} addresses specific properties of Hermite
  polynomials which are required in the proofs of the ensuing DNN
  emulation bounds.  Section~\ref{sec:HermReLU} then contains the core
  results of the present paper: we provide explicit constructions of
  $\ReLU$ and of $\ReLU^k$ DNNs which emulate Hermite polynomials in
  one dimension.  We generalize, via the approximate product operator,
  also to multiple dimensions.  Section~\ref{sec:AnFctRd} then has a
  first application: exponential DNN emulation rate bounds of
  nonlinear, holomorphic maps on $\R^d$, in finite dimension $d$.
  Section~\ref{sec:bxdXHol} addresses the infinite-dimensional case.
  Section~\ref{sec:PDESol} presents an application,
  dimension-independent expression rate bounds for solutions of
  linear, elliptic PDEs with a random coefficient, 
  which is a log-Gaussian random field. The final
  Section~\ref{sec:Concl} reviews the main results, and indicates
  extensions and further applications of the presently developed
  theory.

{\bf Acknowledgement}:
Work performed in part in the programme ``Mathematics of Deep Learning'' (MDL) 
at the Isaac Newton Institute, Cambridge, UK from July-December 2021.
Fertile exchanges, and stimulating workshops are warmly acknowledged.
\section{Hermite polynomials and functions}
\label{sec:HerPol}
  \subsection{Basic definitions and properties}\label{sec:HermPol}
  For $n\in\N_0$ we denote by $H_n$ the $n$th {\bf probabilists'
    Hermite polynomial}\footnote{The \emph{physicists'} Hermite
    polynomials are defined as $x\mapsto H_n(2^{1/2}x)$. Since we
    shall not use them in this manuscript, we simply refer to the
    $(H_n)_{n\in\N_0}$ in the following as the Hermite polynomials.}
  normalized in $L^2(\R,\gamma_1)$, i.e.\
  \begin{equation}\label{eq:Hn}
    H_n(x) := \frac{(-1)^n}{\sqrt{n!}} \e^{\frac{x^2}{2}}\frac{d^n}{dx^n}\e^{-\frac{x^2}{2}} 
    \qquad\forall\,n\in\N_0
  \end{equation}
  with the usual convention $0!=1$. 
Since for any $f\in C^1(\R)$ holds
\begin{equation}\label{eq:ddxfe}
    \frac{d}{dx} \left(f(x)\e^{-\frac{x^2}{2}} \right) = f'(x) \e^{-\frac{x^2}{2}}-x f(x) \e^{-\frac{x^2}{2}},
\end{equation}
it is easy to see that $H_n\in\bbP_n$.

Next, we introduce the {\bf Hermite functions} via
  \begin{equation}\label{eq:hn}
    h_n(x):=\frac{(-1)^n}{\sqrt{\pi^{1/2}2^n n!}}\e^{\frac{x^2}{2}}\frac{d^n}{dx^n}\e^{-x^2}\qquad\forall\,n\in\N_0.
  \end{equation}
The relation between the Hermite polynomials and the Hermite
functions is made clear by the following lemma.
\begin{lemma}\label{lemma:iso1}
    The map
      \begin{equation*}
        \Theta: L^2(\R,\gamma_1) \to L^2(\R): f(x) \mapsto f(2^{1/2}x)\frac{\e^{-\frac{x^2}{2}}}{\pi^{\frac{1}{4}}}
      \end{equation*}
      is an isometric isomorphism and $\Theta(H_n)=h_n$ for all
      $n\in\N_0$.
\end{lemma}
\begin{proof}
    Let $f\in L^2(\R,\gamma_1)$. 
    Using the change of variables $x=2^{1/2}y$
    \begin{equation}
      \norm[L^2(\R,\gamma_1)]{f}^2
      =\int_{\R}f(x)^2 \frac{\e^{-\frac{x^2}{2}}}{{\sqrt{2\pi}}}\dd x
      =\int_{\R}f(2^{1/2}y)^2\frac{\e^{-y^2}}{\pi^{\frac{1}{2}}}\dd y
      =\norm[L^2(\R)]{\Theta(f)}^2.
    \end{equation}
    Thus $\Theta$ is an isometry. 
    By a similar argument
    $\tilde\Theta(F)(y):=
    F(\frac{y}{2^{1/2}})\exp({\frac{y^2}{4}})\pi^{1/4}$ defines an
    isometry from $L^2(\R)\to L^2(\R,\gamma_1)$ and
    $\Theta\circ\tilde\Theta$ is the identity. In all, $\Theta$ is an
    isometric isomorphism.

    To show $\Theta(H_n)=h_n$, denote $r(x):=\e^{-x^2}$ and
    $r(\ii x)=\e^{x^2}$. Due to
    $\frac{d^n}{dx^n}r(\frac{x}{2^{1/2}}) =
    2^{-n/2}r^{(n)}(\frac{x}{2^{1/2}})$ we have by \eqref{eq:Hn}
    \begin{equation*}
      H_n(x) = \frac{(-1)^n}{\sqrt{2^n n!}} r\left(\ii \frac{x}{2^{1/2}}\right) r^{(n)}\left(\frac{x}{2^{1/2}}\right)
    \end{equation*}
    and by \eqref{eq:hn}
    \begin{equation*}
      h_n(x) =
      \frac{(-1)^n}{\sqrt{\pi^{1/2}2^n n!}} \e^{-\frac{x^2}{2}}r(\ii x) r^{(n)}(x).
    \end{equation*}
    Thus
    \begin{equation*}
      h_n(x) = H_n(2^{1/2}x)\frac{\e^{-\frac{x^2}{2}}}{\pi^{\frac{1}{4}}}=\Theta(H_n).\qedhere
    \end{equation*}
  \end{proof}

  As is well-known, these sequences are orthonormal bases in the
  respective spaces. We recall the classical proof for the convenience
  of the reader.

\begin{proposition}\label{prop:ONB2}
  It holds
  \begin{enumerate}
  \item\label{item:HnONB} $(H_n)_{n\in\N_0}$ is an ONB of
    $L^2(\R,\gamma_1)$,
  \item\label{item:hnONB} $(h_n)_{n\in\N_0}$ is an ONB of $L^2(\R)$.
  \end{enumerate}
\end{proposition}
\begin{proof}
  We start by showing orthonormality of $(H_n)_{n\in\N_0}$ in
  $L^2(\R,\gamma_1)$.
  Let $n\in\N_0$, $m\in\N$ and $n\le m$.  Integrating by parts we have
  \begin{equation*}
    \int_\R H_n(x)H_m(x)\dd \gamma_1(x)
    = \frac{(-1)^{n+m}}{\sqrt{2\pi n!m!}}\int_\R H_n(x)\frac{d^m}{dx^m} \e^{-\frac{x^2}{2}}\dd x
    = \frac{(-1)^{n}}{\sqrt{2\pi n!m!}}\int_\R \e^{-\frac{x^2}{2}}\frac{d^m}{dx^m}H_n(x)\dd x.
  \end{equation*}
  Since $H_n\in\bbP_n$, for $n<m$ the integrand vanishes.  In case
  $n=m$, $\frac{d^n}{dx^n}H_n(x)$ equals $n!$ times the leading
  coefficient of $H_n$.  Using \eqref{eq:Hn}-\eqref{eq:ddxfe} one
  obtains $\frac{d^n}{dx^n}H_n(x)=(-1)^nn!$.  Since
  $\int_\R H_0(x)^2\dd\gamma_1(x)=\int_\R 1\dd\gamma_1(x)=1$, we have
  shown
  \begin{equation*}
    \int_\R H_n(x)H_m(x)%
    \dd\gamma_1(x)=\delta_{n,m}
    \qquad \forall\, n,m\in\N_0.
  \end{equation*}

  We show completeness of $(H_n)_{n\in\N_0}$ in $L^2(\R,\gamma_1)$.
  Since $H_n\in\bbP_n$ (with nonzero leading coefficient), it suffices
  to show density of all polynomials in $L^2(\R,\gamma_1)$.  Let
  $f\in L^2(\R,\gamma_1)$ be such that $\int_\R f(x) x^n\dd\gamma_1(x)=0$
  for all $n\in\N_0$.  Define
  $g(z):=\int_\R f(x)\exp(zx)\dd\gamma_1(x)$, which yields an entire
  function on $\C$.  It holds
  $g^{(n)}(0)=\int_\R x^n f(x)\dd\gamma_1(x)=0$ for all $n\in\N$. Thus
  $g\equiv 0$. However, $\R\ni x\mapsto g(-\ii x)$ is the Fourier
  transform of $f$.  This implies $f\equiv 0$ and consequently the
  Hermite polynomials $(H_n)_{n\in\N_0}$ are dense in
  $L^2(\R,\gamma_1)$.

  Finally, since $\Theta:L^2(\R,\gamma_1)\to L^2(\R)$ in Lemma
  \ref{lemma:iso1} is an isometric isomorphism, it transforms the ONB
  $(H_n)_{n\in\N_0}$ of $L^2(\R,\gamma_1)$ to an ONB
  $(\Theta(H_n))_{n\in\N_0}=(h_n)_{n\in\N_0}$ of $L^2(\R)$.
\end{proof}

\subsection{Some preliminary bounds}
\label{sec:PrlBnd}
We will use \emph{Cramer's bound} \cite{indritz} on the Hermite
functions,
\begin{equation}\label{eq:cramer}
  \sup_{x\in\R}|h_n(x)|\le \pi^{-1/4}\qquad \forall n\in\N_0.
\end{equation}

The Hermite polynomials allow the explicit representation, see,
e.g., \cite[Eqn. (5.5.4)]{szego}\footnote{A $\sqrt{n!}$ factor is due to a
different scaling, compare \cite[Eqn. (5.5.3)]{szego} with \eqref{eq:Hn}.}
\begin{equation}\label{eq:Hnexpl}
  H_n(x)=\sum_{j=0}^{\lfloor n/2\rfloor} \frac{\sqrt{n!}(-1)^j}{j!(n-2j)!2^j} x^{n-2j}.
\end{equation}
In the following we also write $H_n(x)=\sum_{j=0}^n c_{n,j}x^n$.
\begin{lemma}\label{lemma:sumcnj}
  For all $n\in\N_0$
  \begin{equation}
    \sum_{j=0}^n |c_{n,j}|\le 6^{n/2} \le 3^n.
  \end{equation}
\end{lemma}
\begin{proof}
  One checks (e.g.\ with Stirling's inequality) that
  $\sqrt{n!}\le 2^n (\lfloor n/2\rfloor)!$ for all $n\in\N_0$. By
  \eqref{eq:Hnexpl} the term $\sum_{j=0}^n |c_{n,j}|$ is bounded by
  \begin{equation*}
    \sum_{j=0}^{\lfloor n/2\rfloor} \frac{\sqrt{n!}}{j!(n-2j)!2^j}
    \le
    2^n \sum_{j=0}^{\lfloor n/2\rfloor}
    \frac{\lfloor \frac{n}{2}\rfloor!}{j!(\lfloor \frac{n}{2}\rfloor -j)!2^j} 
    = 
    2^n \sum_{j=0}^{\lfloor n/2\rfloor}\binom{\lfloor \frac{n}{2}\rfloor}{j} 2^{-j}
    =
    2^n\left(1+\frac 1 2\right)^{\lfloor \frac{n}{2}\rfloor}.
  \end{equation*}
  The last term is bounded by $2^n(3/2)^{n/2} \leq 3^n$ which
  concludes the proof.
\end{proof}
Lemma \ref{lemma:sumcnj} implies the (crude) bound
\begin{equation}\label{eq:Hnbound}
  |H_n(x)|\le (3\max\{1,|x|\})^n\qquad\forall x\in\R.
\end{equation}

In the following for $n\in\N_0\cup\{-1\}$, $n!!$ denotes the double
factorial, i.e.\ $-1!!=0!!=1!!=1$ and $n!!=n\cdot(n-2)!!$ if $n\ge 2$.

\begin{lemma}\label{lemma:intexpxn}
  Let $M\ge 2$ and $n\in\N_0$. Then
  \begin{equation}\label{eq:intexpxn}
    \int_{|x|>M}\e^{-\frac{x^2}{2}}x^n\dd x\le n!! M^n \e^{-\frac{M^2}{2}}.
  \end{equation}
\end{lemma}
\begin{proof}
  Set $a_n:=\int_{x>M}x^n\e^{-\frac{x^2}{2}}\dd x$.
  For $n=0$
  \begin{equation*}
    a_0=\int_{x>M}\e^{-\frac{x^2}{2}}\dd x =
    \int_{y>0}\e^{-\frac{(y+M)^2}{2}}\dd y
    = \e^{-\frac{M^2}{2}}
    \int_{y>0}\e^{-\frac{y^2}{2}-My}\dd y
    \le \frac{1}{2}\e^{-\frac{M^2}{2}},%
  \end{equation*}
  where we used
  \begin{equation*}
    \int_{y>0}\e^{-\frac{y^2}{2}-My}\dd y
    \le \int_{y>0}\e^{-\frac{y^2}{2}-2y}\dd y
    \le \left(\int_{y>0}\e^{-y^2}\dd y\right)^{1/2}
    \left(\int_{y>0}\e^{-4y}\dd y\right)^{1/2}
    =\frac{\pi^{1/4}}{2^{1/2}}\frac{1}{\sqrt{4}}<\frac{1}{2},
  \end{equation*}
  which follows by the well-known fact
  $\int_{y>0}\e^{-y^2}\dd y=\sqrt{\pi}/2$. For $n=1$
  \begin{equation*}
    a_1 = \int_{x>M}x\e^{-\frac{x^2}{2}}\dd x =
    -\int_{x>M}(\e^{-\frac{x^2}{2}})'\dd x = \e^{-\frac{M^2}{2}}.
  \end{equation*}
  This shows \eqref{eq:intexpxn} for $n\in\{0,1\}$. For any $n\ge 2$,
  using integration by parts
  \begin{equation*}
    a_n = \int_{x>M}x^n\e^{-\frac{x^2}{2}}\dd x =
    -\int_{x>M}x^{n-1}(\e^{-\frac{x^2}{2}})'\dd x =
    M^{n-1}\e^{-\frac{M^2}{2}}+(n-1)\int_{x>M}x^{n-2}\e^{-\frac{x^2}{2}}\dd x,
  \end{equation*}
  so that $a_n= M^{n-1}\exp(-M^2/2)+(n-1)a_{n-2}$. %
  For $n\in\{0,1\}$ we have in particular shown
  $a_n\le (n-1)!!M^n\exp(-M^2/2)$. Using $M^{n-1}+M^{n-2}\le M^n$
  since $M\ge 2$, by induction we get
  \begin{equation*}
    \forall n\geq 2: \;\; 
    a_n\le \e^{-\frac{M^2}{2}}M^{n-1}+(n-1)!! M^{n-2}\e^{-\frac{M^2}{2}} 
    \le (n-1)!!M^n\e^{-\frac{M^2}{2}}\;.
  \end{equation*}
  Using this bound and again the recurrence we obtain for $n\ge 2$
  \begin{equation}\label{eq:anfinal}
    a_n\le \e^{-\frac{M^2}{2}}(M^{n-1}+(n-1)!!M^{n-2})
    \le \e^{-\frac{M^2}{2}} M^{n-2}(M+(n-1)!!).
  \end{equation}
  For all $x\ge 1$ holds $M+x\le \frac{3}{4} M^2 x$ because $M\ge 2$. 
  Furthermore $(n-1)!!(3/2)\le n!!$ for all $n\ge 2$.  Hence, with
  $x=(n-1)!!\ge 1$ we get
  $M+(n-1)!!\le 3M^2 (n-1)!!/4\le \frac{1}{2} M^2 n!!$.  Together with
  \eqref{eq:anfinal} this finally implies
  $a_n\le \frac{1}{2}n!!M^n\exp(-M^2/2)$ and concludes the proof.
\end{proof}

We note in passing that Lemma \ref{lemma:intexpxn} and
\eqref{eq:Hnbound} imply for every $M\ge 2$,
$p\ge 1$ and $n\in\N_0$
\begin{equation}\label{eq:intHnxn}
  \int_{|x|>M}|H_n(x)|^p\dd \gamma_1(x)
  =\frac{1}{\sqrt{2\pi}}\int_{|x|>M}|H_n(x)|^p\e^{-\frac{x^2}{2}}\dd x
  \le \frac{1}{\sqrt{2\pi}}(pn)!!(3M)^{pn}\e^{-\frac{M^2}{2}}.
\end{equation}
\section{DNN emulation of Hermite polynomials}
\label{sec:HermReLU}
A key technical step in the DNN expression rate analysis of Gaussian
random fields is the ReLU NN expression of Hermite polynomials.  Due
to general representation of GRFs in terms of Hermite-expansions
(e.g. \cite{CamMart47,Janson,bogachGM} and the references there)
quantitative bounds for ReLU NN expression rates of GRFs will follow
from assumptions on summability of Hermite coefficient sequences of
the GRFs and from ReLU DNN expression rates of Hermite polynomials
$H_n$.  To establish the latter is the purpose of the present section.
Due to the goal of expressing truncated Hermite gpc expansions, our
main result in the present section, Theorem \ref{thm:reluhermite},
will provide quantitative bounds of expression of \emph{(collections
  of tensor products of) Hermite polynomials} by \emph{by one common
  ReLU NN architecture}.

\subsection{Univariate Hermite polynomials}
\label{sec:UniVrH}
We start by recalling that univariate, continuous piecewise linear
functions can be realized exactly by shallow $\ReLU$-NNs, see, e.g.,
\cite[Lemma 4.5]{SZ18}.
\begin{lemma}\label{lemma:pwlin}
  Let $-\infty<x_0<x_1<\dots<x_{n-1}<x_n<\infty$ induce a
  partition of $\R$ into $n+2 \in\N$ intervals. For any continuous
  piecewise linear function $f:\R\to\R$ w.r.t.\ this partition, there
  exist a ReLU NN $\phi:\R\to\R$ such that $\phi(x)=f(x)$ for $x\in\R$
  and ${\size}(\phi) \le 2(n+2)+1$, ${\depth}(\phi)=1$.
\end{lemma}
Next, we address truncation of $\ReLU$-NNs to finite support in $\R$.

  \begin{lemma}\label{lemma:cut}
    Let $M>0$ and let $\phi:\R\to\R$ be a ReLU
    NN. %
    For every $\delta\in (0,M)$ there exists a ReLU NN $\psi:\R\to\R$
    satisfying
    $\sup_{x\in [-M,M]}|\psi(x)|\le \sup_{x\in [-M,M]}|\phi(x)|$,
    \begin{equation}
      \psi(x) = \begin{cases}
        \phi(x) & x\in [-M+\delta,M-\delta]\\
        0 & x\in\R\backslash [-M,M]
      \end{cases}
    \end{equation}
    ${\size}(\psi)\le C(1+{\size}(\phi))$ and
    ${\depth}(\psi)\le C(1+{\depth}(\phi))$, with $C>0$ independent of
    $M$, $\delta$, $\phi$.
  \end{lemma}
  \begin{proof}
    Since $\phi$ is a ReLU NN, there exists $N>0$ such that
    $\phi|_{(-\infty,-N]}$ and $\phi|_{[N,\infty)}$ are linear. We now
    construct a ReLU NN $\eta$ such that
    $\eta|_{[-M,M]}=\phi|_{[-M,M]}$ and $\eta|_{(-\infty,-M]}$ and
    $\eta|_{[M,\infty)}$ are linear.

    Set
    $$p(x)=\frac{\phi(M)+\phi(-M)}{2}+x\frac{\phi(M)-\phi(-M)}{2M},$$
    i.e.\ $p:\R\to\R$ is linear and $p(-M)=\phi(-M)$,
    $p(M)=\phi(M)$. 
    Then $(\phi-p)(M)=(\phi-p)(-M)=0$. 
    For $x\in\R$
    \begin{equation*}
      q(x):=M-\sigma(x)-\sigma(-x) = \begin{cases}
        x+M & x<0\\
        -x+M & x\ge 0\\
      \end{cases}
    \end{equation*}
    is a ReLU NN satisfying $q(-M)=q(M)=0$, $q|_{(-M,M)}>0$, 
    $q|_{(-\infty,-M)}<0$ and $q|_{(M,\infty)}<0$.  
    Since
    $(\phi-p)|_{(-\infty,N]}$ and $(\phi-p)|_{[N,\infty)}$ are linear,
    we can find $\alpha>0$ such that $(\phi-p)+\alpha q$ is positive
    on $(-M,M)$ and negative on $\R\backslash [-M,M]$.  Then
      $$\eta(x) = \sigma(\phi(x)-p(x)+\alpha q(x)) +p(x)-\alpha q(x)$$
      equals $\phi(x)$ for $x\in [-M,M]$ and $\eta|_{(-\infty,-M]}$
      and $\eta|_{[M,\infty)}$ are linear. Since we only added and
      subtracted continuous, piecewise linear functions from 
      $\phi$ and composed them with $\sigma$, the function $\eta$ 
      can be expressed by a ReLU NN.
      
      Now we construct $\psi$. Wlog let $\delta\in (0,M)$ be so small
      that $\eta|_{[-M,-M+\delta]}$ and $\eta|_{[M-\delta,M]}$ are
      linear (which is possible because $\eta$ is a continuous,
      piecewise linear function).  Then both, $\eta|_{[M-\delta,M]}$
      and $\eta|_{[M,\infty)}$ are linear, and by Lemma
      \ref{lemma:pwlin} the function $r:\R\to\R$ that is continuous,
      piecewise linear on the partition $x_0=-\infty$, $x_1=M-\delta$,
      $x_2=M$, $x_3=\infty$ and satisfies $r|_{(-\infty,M-\delta]}=0$,
      $r(M)=\eta(M)$ and $r|_{[M,\infty)}=\eta|_{[M,\infty)}$ is
      expressed by a network of size $O(1)$. Then
      $\eta-r|_{(-\infty,M-\delta]}=\eta|_{(-\infty,M-\delta]}$ and
      $\eta-r|_{[M,\infty)}\equiv 0$. 
      Furthermore
      $(\eta-r)(M-\delta)=\eta(M-\delta)$, $(\eta-r)(M)=0$ and
      $(\eta-r)|_{[M-\delta,M]}$ is linear so that
      $\sup_{x\in [M-\delta,M]}|\eta(x)-r(x)|\le |\eta(M-\delta)|$.
      Similarly, we can construct $s:\R\to\R$ continuous, piecewise
      affine such that $s|_{[-M+\delta,\infty)}\equiv 0$,
      $s(-M)=\eta(-M)$ and $s|_{(-\infty,-M]}=\eta|_{(-\infty,-M]}$.
      Then $\psi=\eta-r-s$ is as claimed.
    \end{proof}
    We are now in position to state our main result on architecture
    and quantitative bounds for emulations of Hermite polynomials by
    deep $\ReLU$-NNs.
    \begin{proposition}\label{prop:Mneps}
      Let $n\in\N_0$, $M>0$ and $\eps\in (0,\e^{-1})$ be arbitrary.
      Then there exists a ReLU NN $\tilde H_{n,M,\eps}:\R\to\R$ such
      that
      \begin{enumerate}
      \item\label{item:Hnerror}
        $\norm[L^2(\R,\gamma_1)]{H_n-\tilde H_{n,M,\eps}}\le \eps+
        \sqrt{2n!!} (3M)^n \e^{-\frac{M^2}{4}}$,
      \item\label{item:Hntype} $\tilde H_{n,M,\eps}(x)=0$ 
        for $|x|>M$ and $\sup_{x\in\R}|\tilde H_{n,M,\eps}(x)|\le 1+(3M)^n$,
      \item\label{item:Hnsize} for a constant
        $C>0$ independent of $n$, $M$, $\eps$
        \begin{align*}
          {\size}(\tilde H_{n,M,\eps})
          &\le C\Big(1+n^2\log(M)+n\log\Big(\frac{n}{\eps}\Big)\Big),\nonumber\\
          {\depth}(\tilde H_{n,M,\eps})&\le
                                         C((1+\log(n))(n\log(M)-\log(\eps))).
        \end{align*}
      \end{enumerate}
    \end{proposition}
    \begin{proof}
      In this proof we will need the following result shown in
      \cite[Prop.~4.2]{OPS19}: for any polynomial $p(x)=\sum_{j=0}^n
      c_j x^j$, there exists a neural network $\tilde
      p$ such that $|p(x)-\tilde p(x)|\le \eps$ for all
      $x\in[-1,1]$, and with
      $C_0:=\max\{2,\sum_{j=0}^n|c_j|\}$ it holds
      \begin{equation*}
        {\size}(\tilde p)\le C\Big((1+n)\log\Big(\frac{C_0}{\eps}\Big)+n\log(n)\Big),\qquad
        {\depth}(\tilde p)\le
        C\Big((1+\log(n))\log\Big(\frac{C_0}{\eps}\Big)+\log(n)^3\Big),
      \end{equation*}
      where the constant $C$ is independent of $\eps\in
      (0,\e^{-1})$ and of $n\in\N_0$.

      Denote by
      $H_{n,M}(x)=H_n(Mx)$ the rescaled Hermite polynomial.  
      Then $H_{n,M}(x)=\sum_{j=0}^n c_{n,j}M^jx^j$.
      By Lemma \ref{lemma:sumcnj} it holds
      $C_0:=\sum_{j=0}^n|c_{n,j}M^j|\le M^n\sum_{j=0}^n|c_{n,j}|\le
      (3M)^n$. Thus by \cite[Prop.~4.2]{OPS19} there exists a neural
      network $\hat H_{n,M,\eps}$ such that
      \begin{equation}\label{eq:Hnmerr}
        \sup_{x\in [-1,1]}|H_{n,M}(x)-\hat H_{n,M,\eps}(x)|\le \frac{\eps}{2}
      \end{equation}
      and
      \begin{align}\label{eq:tHnMsize}
        {\size}(\hat H_{n,M,\eps})
      & \le C\Big(1+n^2\log(M)+n\log\Big(\frac{n}{\eps}\Big)\Big),\nonumber\\
        {\depth}(\hat H_{n,M,\eps}) 
      &\le C\Big((1+\log(n))(n\log(M)-\log(\eps))\Big),
      \end{align}
      for some constant $C>0$ independent of $M\ge 2$,
      $n\in\N_0$ and $\eps\in
      (0,\e^{-1})$ (for the bound on the depth we could absorb the
      term $\log(n)^3$ in $n\log(M)$, since $\log(M)>0$ due to $M\ge
      2$).

      With $\delta>0$ for the moment fixed, but to be chosen shortly
      (in dependence of $n$ and $\eps$), by Lemma \ref{lemma:cut}
      there exists a NN $\tilde H_{n,M,\eps}$ such that
      $\sup_{x\in [-M,M]}|\tilde H_{n,M,\eps}(x)|\le \sup_{x\in [-1,1]}|\hat H_{n,M,\eps}(x)|$ 
      and
      \begin{equation*}
        \tilde H_{n,M,\eps}(x):=\begin{cases}
          \hat H_{n,M,\eps}(\frac x M)& x\in[-M+\delta,M-\delta]\\
          0 & x\in \R\backslash[-M,M].
        \end{cases}
      \end{equation*}
      For $x\in[-M+\delta,M-\delta]$
      \begin{equation*}
        |H_{n}(x)-\tilde H_{n,M,\eps}(x)|
        =
        \left|H_{n,M}\Big(\frac x M\Big) - \hat H_{n,M,\eps}\Big(\frac x M\Big)\right|\le\eps.
      \end{equation*}
      By Lemma \ref{lemma:cut}, the depth and size bounds for
      $\hat H_{n,M,\eps}$ from \eqref{eq:tHnMsize} are also valid for
      $\tilde H_{n,M,\eps}$ (possibly for a different constant $C$),
      which shows \ref{item:Hnsize}.

      Next, for $x\in \R\backslash[-M,M]$ it holds
      $|H_n(x)-\tilde H_{n,M,\eps}(x)|=|H_n(x)|$.  By
      \eqref{eq:Hnbound}, \eqref{eq:Hnmerr} and Lemma \ref{lemma:cut}
      we find
      \begin{equation}\label{eq:tHnbound}
        \sup_{x\in\R}|\tilde H_{n,M,\eps}(x)| 
      \le 
        \sup_{x\in [-1,1]}|\hat H_{n,M,\eps}(x)|
        \le \eps + \sup_{x\in [-M,M]}|H_{n}(x)|
        \le 1 + (3M)^n,
      \end{equation}
      and thus for $x\in [-M,M]\backslash [-M+\delta,M-\delta]$ we get
      $|H_n(x)-\tilde H_{n,M,\eps}(x)| \le 1+2(3M)^n$. 
      Hence using
      \eqref{eq:intHnxn}
      \begin{align*}
        \norm[L^2(\R,\gamma_1)]{H_n-\tilde H_{n,M,\eps}}
        &\le \norm[{L^2([-M,M],\gamma)}]{H_n-\tilde H_{n,M,\eps}}
          +\norm[{L^2([-M,M]\backslash [-M+\delta,M-\delta],\gamma)}]{H_n-\tilde H_{n,M,\eps}}
          \nonumber\\
        &\quad+\norm[{L^2(\R\backslash[-M,M],\gamma)}]{H_n}\nonumber\\
        &\le \frac{\eps}{2} + \sqrt{\delta}(1+2(3M)^n)
          + \sqrt{(2n)!!} (3M)^n \e^{-\frac{M^2}{4}}.
      \end{align*}
      Choosing $\delta>0$ small enough it holds
      $\sqrt{\delta}(1+2(3M)^n)\le \frac{\eps}{2}$, which shows
      \ref{item:Hnerror}.
      
      Finally, \ref{item:Hntype} holds by \eqref{eq:tHnbound} and the
      construction of $\tilde H_{n,M,\eps}$.
    \end{proof}

    \begin{lemma}
      For all $n\in\N$ it holds
      $\sup_{x\in\R}x^n\e^{-\frac{x^2}{2}}=n^{\frac{n}{2}}\e^{-\frac{n}{2}}=\e^{\frac{n(\log(n)-1)}{2}}$.
    \end{lemma}
    \begin{proof}
      We have
      $(x^n\e^{-x^2/2})'=(nx^{n-1}-x^{n+1})\e^{-x^2/2}
      =x^{n-1}(n-x^{2})\e^{-x^2/2}$. The only positive root of this
      term is $x=\sqrt{n}$, which implies the lemma.
    \end{proof}

    \begin{corollary}\label{cor:Mneps}
      Consider the setting of Prop.~\ref{prop:Mneps} and set, for
      $\eps\in (0,\e^{-1})$,
      \begin{equation}\label{eq:Mneps}
        M(n,\eps) := \sqrt{24(n\log(2n)-\log(\eps))}, \quad n\in \N.
      \end{equation}
      With this choice of $M$, define the $\ReLU$-NN
      $\tilde H_{n,\eps}:=\tilde H_{n,M,\eps}:\R\to\R$.  It satisfies
      \begin{enumerate}
      \item\label{item:Hnerror2}
        $\norm[L^2(\R,\gamma_1)]{H_n-\tilde H_{n,\eps}}\le 2\eps$,
      \item\label{item:Hntype2} $\tilde H_{n,\eps}(x)=0$ for $|x|>M$
        and $\sup_{x\in\R}|\tilde H_{n,\eps}(x)|\le 1+(3M)^n$,
      \item\label{item:Hnsize2} for some $C>0$ independent of $n$ and
        $\eps$
        \begin{align*}
          {\size}(\tilde H_{n,\eps})&\le C
                                      \Big(1+n^2(\log(n)+\log(-\log(\eps)))+n\log\Big(\frac{n}{\eps}\Big)\Big),\nonumber\\
          {\depth}(\tilde H_{n,\eps})&\le
                                       C(1+n\log(n)^2+n\log(n)\log(-\log(\eps))-\log(n)\log(\eps)).
        \end{align*}
      \end{enumerate}
    \end{corollary}
    \begin{proof}
      Inserting $M$ from \eqref{eq:Mneps} into the bound in
      Prop.~\ref{prop:Mneps} \ref{item:Hnsize} we get
      \begin{equation*} {\size}(\tilde H_{n,\eps})\le
        C\Bigg(1+\frac{1}{2}n^2\log\big(24(n\log(2n)-\log(\eps))\big)+n\log\Big(\frac{n}{\eps}\Big)\Bigg).
      \end{equation*}
      For $a$, $b\ge 1$ it holds
      \begin{equation*}
        \log(a+b)=\log(a)+\int_{a}^{a+b}\frac{1}{x}\dd x\le
        \log(a)+\int_1^{1+b}\frac{1}{x}\dd x = \log(a)+\log(1+b)
        \le 1 + \log(a)+\log(b).
      \end{equation*}
      With $a=24 n\log(2n)$ and $b=-24 \log(\eps)$ we get
      \begin{align*}
        {\size}(\tilde H_{n,\eps})&\le
                                    C\Bigg(1+\frac{1}{2}n^2\big(1+\log({24}n)+\log(\log(2n))+\log(-{24}\log(\eps))\big)+n\log\Big(\frac{n}{\eps}\Big)\Bigg)\nonumber\\
                                  &\le C \Big(1+n^2(\log(n)+\log(-\log(\eps)))+n\log\Big(\frac{n}{\eps}\Big)\Big)
      \end{align*}
      for a constant $C$ independent of $n$ and $\eps$.  This shows
      the bound on the size in \ref{item:Hnsize}. The bound on the
      depth is obtained similarly.

      To show \ref{item:Hnerror2} we use Prop.~\ref{prop:Mneps}
      \ref{item:Hnerror} and claim that
      $\sqrt{(2n)!!} (3M)^n \e^{-\frac{M^2}{4}}\le\eps$.  Since
      $\sqrt{(2n)!!}\le (2n)^n$ it is sufficient to show that
      \begin{equation}\label{eq:claimM}
        -\frac{M^2}{4}+n\log(2n)+n\log(3M)-\log(\eps)\le 0.
      \end{equation}
      The definition of $M$ implies $\frac{M^2}{24}\ge n\log(2n)-\log(\eps)$
      and thus
      \begin{equation}\label{eq:claimM1}
        -\frac{M^2}{24}+n\log(2n)-\log(\eps)\le 0.
      \end{equation}
      Next we show
      \begin{equation*}
        -\frac{M^2}{6}+n\log(3M)\le 0,
      \end{equation*}
      which will then imply \eqref{eq:claimM} due to
      $\frac{1}{24}+\frac{1}{6}\le \frac{1}{4}$. 
      The last inequality is equivalent to $\frac{M^2}{\log(3M)}\ge 6 n$. 
      The function
      $x\mapsto \frac{x^2}{\log(3x)}$ is monotonically increasing for $x\ge 1$
      and by \eqref{eq:Mneps} it holds $M\ge \sqrt{24n\log(2n)}=:x$.
      Hence
      \begin{equation*}
        \frac{M^2}{\log(3M)}\ge \frac{x^2}{\log(3x)}
        = 6n \frac{4\log(2n)}{\log(3\sqrt{24n\log(2n)})}.
      \end{equation*}
      It suffices to show that
      $\frac{4\log(2n)}{\log(3\sqrt{24n\log(2n)})}\ge 1$ for all
      $n\in\N$. It is checked directly that this holds for $n=1$, and
      that this term is monotonically increasing for $n\ge 1$, so that
      it is true for all $n\in\N$. Together with \eqref{eq:claimM1}
      this verifies \eqref{eq:claimM}. 
      In all, together with
      Prop.~\ref{prop:Mneps} \ref{item:Hnerror} we get
      $\norm[L^2(\R,\gamma_1)]{H_n-\tilde H_{n,\eps}}\le 2\eps$.
    \end{proof}
    \subsection{Multivariate Hermite polynomials}
    \label{sec:MltVrH}
    We proceed to show $\ReLU$-NN expression bounds for multivariate,
    tensorized Hermite polynomials.

    Recall that
    $\cF = \{ \bsnu\in \N_0^\infty: |\bsnu|<\infty \}$ is
    the set of all
    finitely supported multiindices. %
    For a finite index set $\Lambda \subseteq \cF$,
    we define
    \begin{equation}\label{eq:suppLambda}
      \supp\Lambda:=\set{j\in\supp\bsnu}{\bsnu\in\Lambda}
    \end{equation}
    and
    we introduce the \emph{maximum order} $m(\Lambda)$ and the
    \emph{effective dimension} $d(\Lambda)$ of $\Lambda$ as
    \begin{equation}\label{eq:md}
      m(\Lambda):=\max_{\bsnu\in\Lambda}|\bsnu|_1,\qquad
      d(\Lambda):=\sup_{\bsnu\in\Lambda}|\bsnu|_0.
    \end{equation}

\begin{proposition}[{\cite[Proposition 3.3]{SZ18}}]\label{prop:mul}
  For any $\eps\in (0,\e^{-1})$,
  for every $d\in\N$ and every $A>0$, there exists a $\ReLU$-NN
  $\tilde \prod_{d,A,\eps}:[-A,A]^d\to\R$ such that
  \begin{equation}
    \label{eq:prodabserr}
    \sup_{(x_i)_{i=1}^d\in [-A,A]^d}\left|\prod_{j=1}^d x_j - \tilde\prod_{d,A,\eps}(x_1,\dots,x_d)\right| \le \eps.
  \end{equation}
  There exists a constant $C$ independent of $\eps\in (0,\e^{-1})$,
  $d\in \N$ and $A\ge 1$ such that %
  \begin{equation}
    \label{eq:prodsizedepth}
    \size\Big(\tilde\prod_{d,A,\eps}\Big)\le C\Big(1+d\log\Big(\frac{dA^d}{\eps}\Big)\Big)
    \quad\text{and}\quad
    \depth\Big(\tilde\prod_{d,A,\eps}\Big)\le C\Big(1+\log(d)\log\Big(\frac{dA^d}{\eps}\Big)\Big).
  \end{equation}
\end{proposition}

    \begin{theorem}\label{thm:reluhermite}
      Let $\Lambda\subseteq\CF$ be finite and downward closed.  Then
      for every $\eps\in (0,\e^{-1})$ there exists a neural network
      $\Phi=\{\tilde
      H_{\eps,\bsnu}\}_{\bsnu\in\Lambda}:\R^{|\supp\Lambda|}\to\R^{|\Lambda|}$
      such that %
      \begin{equation*}
        \max_{\bsnu\in\Lambda} \;\;
        \norm[L^2(U,\gamma)]{H_\bsnu-\tilde H_{\eps,\bsnu}}\le \eps,
      \end{equation*}
      and there exists a positive constant $C$ (independent of
      $m(\Lambda)$, $d(\Lambda)$ and of $\eps\in (0,\e^{-1})$) such that
      \begin{align*}
        {\size}(\Phi)&\le C |\Lambda|m(\Lambda)^3\log(1+m(\Lambda))d(\Lambda)^2\log(\eps^{-1}),
        \\
        {\depth}(\Phi)
                     &\le C m(\Lambda)\log(1+m(\Lambda))^2d(\Lambda)\log(1+d(\Lambda))\log(\eps^{-1}).
      \end{align*}
    \end{theorem}
    \begin{proof}
      Fix $\eps\in (0,\e^{-1})$.  Throughout this proof, %
        we write $m:=m(\Lambda)$, $d:=d(\Lambda)$
        and we assume w.l.o.g.\ that $m\ge 1$ and $d\ge 1$ (otherwise
        $\Lambda=\emptyset$ or $\Lambda=\{\bsnul\}$, and these cases
        are trivial). Furthermore, with the constant $M$ as defined in
        \eqref{eq:Mneps}, set
      \begin{equation*}
        A:=1+(3M)^m, \quad \tilde\eps:=\eps 2^{-d}.
      \end{equation*}

      {\bf Step 1.} We define $\tilde H_{\eps,\bsnu}$ and show that
      $\norm[L^2(U,\gamma)]{H_\bsnu-\tilde H_{\eps,\bsnu}} \le \eps$.

      Let $H_\bsnul:= 1$ and for $\bsnul\neq \bsnu\in\Lambda$
      \begin{equation*}
        \tilde H_{\tilde \eps,\bsnu}:=\tilde\prod_{|\bsnu|_0,A,\eps}
        ((\tilde H_{\tilde \eps,\nu_j}(x_j))_{j\in\supp\bsnu}).
      \end{equation*}
      Then for $\bsnul\neq\bsnu\in\Lambda$
      \begin{align}\label{eq:H-tH}
        \norm[L^2(U,\gamma)]{H_\bsnu-\tilde H_{\eps,\bsnu}}
        &\le \normc[L^2(U,\gamma)]{H_\bsnu-\prod_{j\in\supp\bsnu}\tilde H_{\tilde \eps,\nu_j}}\nonumber\\
        &\quad+ \normc[L^2(U,\gamma)]{\prod_{j\in\supp\bsnu}\tilde H_{\tilde \eps,\nu_j}-\tilde 
          \prod_{|\bsnu|_0,A,\eps}\big((\tilde H_{\tilde\eps,\nu_j})_{j\in\supp\bsnu}\big)}.
      \end{align}
      By Cor.~\ref{cor:Mneps} \ref{item:Hntype2} it holds
      $\sup_{x\in\R}|\tilde H_{\tilde \eps,\nu_j}(x)|\le 1+(3M)^{\nu_j}\le A$ for all $j\in\supp\bsnu$.
      Prop.~\ref{prop:mul} thus implies
      \begin{equation*}
        \left|\prod_{j\in\supp\bsnu}\tilde H_{\tilde \eps,\nu_j}(y_j)-
          \tilde \prod_{|\bsnu|_0,A,\eps}\big((\tilde H_{\tilde\eps,\nu_j}(y_j))_{j\in\supp\bsnu}\big)
        \right|\le\eps
      \end{equation*}
      for all $\bsy\in U$ so that
      $\norm[L^2(U,\gamma)]{\prod_{j\in\supp\bsnu}\tilde
        H_{\nu_j}-\tilde H_{\eps,\bsnu}}\le \eps$. To bound the first
      term in \eqref{eq:H-tH} we compute
      \begin{align*}
        \normc[L^2(U,\gamma)]{\prod_{j\in\supp\bsnu}H_{\nu_j}-\prod_{j\in\supp\bsnu}\tilde H_{\tilde \eps,\nu_j}}
        &\le \sum_{j\in\supp\bsnu} \, \prod_{\substack{i\in\supp\bsnu\\ i<j}}\norm[L^2(\R,\gamma_1)]{H_{\nu_i}}\nonumber\\
        &\quad\cdot \norm[L^2(\R,\gamma_1)]{H_{\nu_j}-\tilde H_{\tilde \eps,\nu_j}}\cdot
          \prod_{\substack{i\in\supp\bsnu\\ i>j}}\norm[L^2(\R,\gamma_1)]{\tilde H_{\tilde \eps,\nu_i}}.
      \end{align*}
      For all $i$ it holds $\norm[L^2(\R,\gamma_1)]{H_{\nu_i}}=1$,
      $\norm[L^2(\R,\gamma_1)]{H_{\nu_i}-\tilde H_{\tilde \eps,\nu_i}}\le\tilde\eps$ 
      by Cor.~\ref{cor:Mneps} and
      thus
        $\norm[L^2(\R,\gamma_1)]{\tilde H_{\tilde \eps,\nu_i}}\le 1 +
        \tilde\eps\le \frac{3}{2}$ (since
        $\tilde\eps\le \eps<\e^{-1}$). 
      Hence
      \begin{equation*}
        \normc[L^2(U,\gamma)]{\prod_{j\in\supp\bsnu}H_{\nu_j}-\prod_{j\in\supp\bsnu}\tilde H_{\tilde \eps,\nu_j}}
         \le |\bsnu|_0\tilde \eps (1+\tilde \eps)^{|\bsnu|_0-1}
         \le \tilde \eps d\Big(\frac{3}{2}\Big)^{d-1}\le \tilde \eps 2^d
         \le \eps,
      \end{equation*}
      where we used $d(\frac{3}{2})^{d-1}\le 2^d$ for all $d\in\N$.

      {\bf Step 2.} We construct
      $\Phi=\{\tilde H_{\eps,\bsnu}\}_{\bsnu\in\Lambda}$ and provide
      bounds on the size and depth of $\Phi$.

      Let $\Phi_1:\R^{|\supp\Lambda|}\to \R^{m|\Lambda|}$, with output
      \begin{equation}\label{eq:Phi1}
        \Phi_1(\bsy) = \Big\{\tilde H_{\tilde\eps,j}(y_i)\Big\}_{i\in\supp\Lambda,j\in\{1,\dots,m\}}.
      \end{equation}
      By Cor.~\ref{cor:Mneps} for each $j\le m$
      \begin{align}\label{eq:sizetHj}
        \size(\tilde H_{\tilde\eps,j})
        &\le C\Big(1+j^2\big(\log(j)+\log(-\log(\tilde\eps))\big)+j\log\Big(\frac{j}{\tilde\eps}\Big)\Big)
          \nonumber\\
        &= C\Bigg(1+j^2\big(\log(j)+\log(d\log(2)-\log(\eps))\big)+j\Big(\log(j)+d\log(2)-\log(\eps)\Big)\Bigg)
          \nonumber\\
        &\le C\Bigg(1+m^2\log(m)+m^2\log(d)+m^2\log(-\log(\eps))+md-m\log(\eps)\Bigg)
          \nonumber\\
        &=:C C_0(m,d,\eps),
      \end{align}
      with $C_0(m,d,\eps)$ denoting the term in brackets, and $C$
      being a constant independent of $m$, $d$ and $\eps$.
      Note that 
      $\log(-\log(\eps))$ is well defined since $-\log(\eps)>1$ due to $\eps<\e^{-1}$.

      To derive a bound on the depth, we observe that by Cor.~\ref{cor:Mneps}
      \begin{align}\label{eq:depthtHj}
        {\depth}(\tilde H_{\tilde\eps,j})
        &\le C(1+j\log(j)^2+j\log(j)\log(-\log(\tilde\eps))-\log(j)\log(\tilde\eps))\nonumber
        \\
        &\le C
          \Big(1+m\log(m)^2+m\log(m)\log(-\log(\eps))+m\log(m)\log(d)-\log(m)\log(\eps)\Big)\nonumber
        \\
        &=:CC_1(m,d,\eps),
      \end{align}
      where $C_1(m,d,\eps)$ is the term in parentheses, and $C$ is a
      positive constant that is independent of $m$, $d$ and $\eps$.
      Concatenating $\tilde H_{\tilde\eps,j}$ with
      $CC_1(m,d,\eps)-\depth(\tilde H_{\tilde\eps,j})$ times the
      identity network $x=\sigma(x)-\sigma(-x)$, we may and will
      assume that each $\tilde H_{\tilde\eps,j}(y_i)$ in
      \eqref{eq:Phi1} has the same depth $C C_1(m,d,\eps)$, and the
      size is bounded by
      $C C_0(m,d,\eps)+ C C_1(m,d,\eps)\le C C_0(m,d,\eps)$ for a
      suitable constant $C$ that is independent of $m$, $d$ and
      $\eps$.

      Next, we let $\Phi_2:\R^{m|\supp\Lambda|}\to\R^{|\Lambda|}$ be
      the network
      \begin{equation}\label{eq:Phi2}
        \Phi_2:=\Big\{\tilde\prod_{|\bsnu|_0,A,\eps}\Big\}_{\bsnu\in\Lambda}.
      \end{equation}
      Then
      \begin{equation*}
        \Phi_2\circ\Phi_1 =\Big\{\tilde\prod_{|\bsnu|_0,A,\eps}\big((\tilde H_{\eps/2^d,\nu_j})_{j\in\supp\bsnu}\big)\Big\}_{\bsnu\in\Lambda} =\Phi.
      \end{equation*}
      It remains to estimate the size and depth of $\Phi_2$. By
      Prop.~\ref{prop:mul}
      \begin{align*}
        \size\Big(\tilde\prod_{|\bsnu|_0,A,\eps}\Big)
        &\le C(1+|\bsnu|_0\log(|\bsnu|_0)+|\bsnu|_0^2\log(A)-|\bsnu|_0\log(\eps))\nonumber\\
        &\le C(1+d\log(d)+d^2\log(A)-d\log(\eps)).
      \end{align*}
      By definition of $A$ and $M$, using $\log(1+x)\le 1+\log(x)$ for $x\ge 1$,
      \begin{equation*}
        \log(A)
        \le 1+m\log(3\sqrt{24(m\log(2m)-\log(\eps))}) 
        \le C(1+m\log(m)+m\log(-\log(\eps))).
      \end{equation*}
      Hence
      \begin{equation}\label{eq:sizeprod}
        \size\Big(\tilde\prod_{|\bsnu|_0,A,\eps}\Big)\le
        C\Big(1+d\log(d)+d^2m\log(m)+d^2m\log(-\log(\eps)) - d\log(\eps)\Big) 
        =:C D_0(m,d,\eps).
      \end{equation}
      In addition, by Prop.~\ref{prop:mul}
      \begin{align}\label{eq:depthprod}
        \depth\Big(\tilde\prod_{|\bsnu|_0,A,\eps}\Big)
        &\le C\Big(1+\log(d)(\log(d)+d\log(A)-\log(\eps))\Big)\nonumber\\
        &\le C\Big(1+\log(d)^2+d\log(d)m\log(m)+d\log(d)m \log(-\log(\eps))-\log(d)\log(\eps)\Big)\nonumber\\
        &=:CD_1(m,d,\eps).
      \end{align}
      Similar as before, by concatenating
      $\tilde\prod_{|\bsnu|_0,A,\eps}$ a suitable number of times with
      the identity network $x=\sigma(x)-\sigma(-x)$, we can assume
      that all networks $\tilde\prod_{|\bsnu|_0,A,\eps}$,
      $\bsnu\in\Lambda$, have the same depth, and a uniform bound on
      the size given by \eqref{eq:sizeprod}.

      We now sum the size of all subnetworks. 
      First note that the
      downward closedness of $\Lambda$ implies
      $|\supp\Lambda|\le|\Lambda|$ (cf.~\eqref{eq:suppLambda}).
      
      Hence, by \eqref{eq:Phi1}, \eqref{eq:sizetHj} and
      \eqref{eq:Phi2}, \eqref{eq:sizeprod} it follows that there
      exists a constant $C>0$ such that for all $0<\eps<\e^{-1}$
      holds
      \begin{align*}
        \size(\Phi)&\le C(1+\size(\Phi_1)+\size(\Phi_2))\nonumber\\
                   &\le C(1+(|\supp\Lambda|m)C_0(m,d,\eps) + |\Lambda|D_0(m,d,\eps))\nonumber\\
                   &\le C |\Lambda|\Big(\big(1+m^3(\log(m)+d)+m^3\log(\eps)\big)
                     +
                     \big(d^2m\log(m)\log(\eps^{-1})\big)
                     \Big)\nonumber\\
                   &\le C (1+|\Lambda|m^3\log(1+m)d^2\log(\eps^{-1})).
      \end{align*}
      Similarly by \eqref{eq:depthtHj} and \eqref{eq:depthprod}
      \begin{align*}
        \depth(\Phi)&\le C (1+\depth(\Phi_1)+\depth(\Phi_2))\nonumber\\
                    &\le C(1+C_1(m,d,\eps)+D_1(m,d,\eps))\nonumber\\
                    &\le C(1+d\log(1+d)m\log(1+m)^2\log(\eps^{-1})).\qedhere
      \end{align*}
    \end{proof}
\begin{remark}\label{rmk:HnRePU}
    The preceding analysis was based on \emph{approximating}
    Hermite polynomials by $\ReLU$-NNs.
      The so-called ``polynomial $\ReLU$'' activation $\ReLU^k$,
      sometimes also referred to as ``rectified power unit''
      (``RePU''), is capable of \emph{exactly} expressing multivariate polynomials,
       i.e. \emph{without emulation error}.
    For an integer $k\geq 2$, this activation function is given by
    $\ReLU^k(x) := \max\{x,0\}^k$. Evidently,
    $\ReLU^k \in W^{k,\infty}_{\rm loc}(\R)$, so that the resulting
    DNNs will inherit this regularity in the input-output maps arising
    as their realizations. From \cite[Prop. 2.14]{OSZ19}, we have the
    following statement.

    Fix $d\in\N$ and $k\in\N$, $k\ge 2$ arbitrary.  Then there exists
    a constant $C>0$ (depending on $k$ but independent of $d$) such
    that for any finite, downward closed $\Lambda\subseteq\N_0^d$ and
    for any $p\in\bbP_\Lambda$ there is a $\ReLU^k$-NN
    $\tilde{p}:\R^d\to\R$ which realizes $p$ exactly and such that
    $\size(\tilde p)\le C|\Lambda|$ and
    $\depth(\tilde p)\le C \log(|\Lambda|)$.
\end{remark}
     \section{DNN approximation of analytic functions in
       $L^2(\R^d,\gamma_d)$}
     \label{sec:AnFctRd}
     In this section, we show that certain analytic functions
     $f\in L^2(\R^d,\gamma_d)$ with finite $d\in\N$ can be approximated
     at an exponential rate by $\ReLU$-NNs.  To state the precise
     assumption on $f$, for $\kappa>0$ introduce the complex open
     strip
     \begin{subequations}
       \begin{equation}\label{eq:Skappa}
         S_\kappa:=\set{z=x+\ii y\in\C}{|y|<\kappa}\subset\C
       \end{equation}
       and for $\bstau=(\tau_j)_{j=1}^d\in (0,\infty)^d$ let
       \begin{equation}\label{eq:Stau}
         S_\bstau := \otimes_{j=1}^dS_{\tau_j}\subset\C^d.
       \end{equation}
     \end{subequations}

    \begin{assumption}\label{ass:analytic}
      There exists $\bstau\in (0,\infty)^d$ so that $f:S_\bstau\to\C$
      is holomorphic. For every $\bsnul \le \bsbeta \le \bstau$ there
      exists $B(\bsbeta)>0$ such that for all
      $\bsx+\ii\bsy = (x_j + \ii y_j)_{j=1}^d \in S_\bsbeta$ it holds
      \begin{equation}\label{eq:f}
        |f(\bsx+\ii \bsy)|
        \le B(\bsbeta) \exp\left(\sum_{j=1}^d\left(\frac{x_j^2}{4}-2^{-1/2}|x_j|(\beta_j^2-\frac{1}{2}y_j^2)^{1/2} \right) \right).
      \end{equation}
    \end{assumption}
    Condition \eqref{eq:f} is a growth condition on $f$ on the
    cylinder $S_\bstau\subset \C^d$.  
    It states that $f$ should increase along the real axis in $x_j$ slower than
    $\exp(\frac{x_j^2}{4})$. The parameters $\beta_j$ quantify this
    further, and will determine the rate of convergence.  The
    occurence of the factor $\exp(\frac{x_j^2}{4})$ stems from the
    fact, that we wish to approximate $f$ in $L^2(\R^d,\gamma_d)$,
    where the Gaussian $\gamma_d$ has Lebesgue density
    $(2\pi)^{-d/2}\exp(-\sum_{j=1}^d \frac{x_j^2}{2})$.  Hence $f$
    increasing faster than $\exp(\frac{x_j^2}{4})$ 
    would imply $f\notin L^2(\R^d,\gamma_d)$.

    \subsection{Polynomial approximation}
    \label{sec:finHerm}
    Recall that the Hermite functions $(h_n)_{n\in\N_0}$ in
    \eqref{eq:hn} form an ONB of $L^2(\R)$.  Our analysis in the
    finite dimensional case is based on the classical paper
    \cite{MR871} of E.\ Hille.
    \begin{theorem}[{\cite[Theorem 1]{MR871}}]\label{thm:hille}
      Let $\tau>0$ and let $g:S_\tau\to\C$ be holomorphic and satisfy:
      for every $\beta\in (0,\tau)$ exists $B(\beta)<\infty$ such that
      for all $x+\ii y\in S_\beta$
      \begin{equation}\label{eq:asshille}
        |g(x+\ii y)|\le B(\beta)\exp(-|x|(\beta^2-y^2)^{1/2}).
      \end{equation}

      Then for every $\beta\in (0,\tau)$ exists a constant $K(\beta)$
      depending on $\beta$ (but independent of $g$, $\tau$ and $n$)
      such that for every $n\in\N$
      \begin{equation}\label{eq:fnbound}
        \left|\int_\R g(x)h_n(x)\dd x\right|\le K(\beta) B(\beta)\exp(-\beta (2n+1)^{1/2}).
      \end{equation}
    \end{theorem}
    We recall part of the proof of the theorem in Appendix \ref{app:hille}.
    The reason is that the result in \cite[Theorem 1]{MR871} does not
    explicitly state the dependence of the occurring constants.  In
    the following we wish to repeatedly apply \eqref{eq:fnbound}
    coordinatewise to obtain a multivariate version.  To this end we
    need \eqref{eq:fnbound} to hold for some $K(\beta)$, where
    $K(\beta)$ is only a function of $\beta$ but does not depend on
    $g$.
  
    To state the multivariate version of Thm.~\ref{thm:hille}, with
    $H_n$ and $h_n$ from \eqref{eq:Hn}, \eqref{eq:hn}, for all
    $\bsnu\in\N_0^d$ in the following
    \begin{equation*}
      H_\bsnu(\bsx):=\prod_{j=1}^d H_{\nu_j}(x_j)\qquad\text{and}\qquad
      h_\bsnu(\bsx):=\prod_{j=1}^d h_{\nu_j}(x_j).
    \end{equation*}
    Moreover we use standard multivariate notation such as
    $(0,\bstau)$ to denote the cube
    $\times_{j=1}^d(0,\tau_j)\subset\R^d$ for
    $\bstau=(\tau_j)_{j=1}^d\in (0,\infty)^d$.

    \begin{corollary}\label{cor:hillemulti}
      Let $d\in\N$, $\bstau\in (0,\infty)^d$ and let $F:S_\bstau\to\C$
      be holomorphic and satisfy: for every $\bsbeta\in (0,\bstau)$
      exists $B(\bsbeta)>0$ such that for all
      $\bsx +\ii \bsy = (x_j + \ii y_j)_{j=1}^d \in S_\bsbeta$
      \begin{equation}\label{eq:assmulti}
        |F(\bsx+\ii \bsy)|\le B(\bsbeta)\exp\left(-\sum_{j=1}^d|x_j|(\beta_j^2-y_j^2)^{1/2}\right).
      \end{equation}

      With $K(\beta_j)>0$
      as in Thm.~\ref{thm:hille} then holds for every $\bsnu\in\N_0^d$
      and every $\bsbeta\in (0,\bstau)$
      \begin{equation}\label{eq:fnbound2}
        |\dup{F}{h_\bsnu}_{L^2(\R^d)}|
        \le 
        \left(\prod_{j=1}^d K(\beta_j)\right) B(\bsbeta)\exp\left(-\sum_{j=1}^d\beta_j (2\nu_j+1)^{1/2}\right).
      \end{equation}      
    \end{corollary}
    \begin{proof}
      Fix $\bsnu\in\N_0^d$ and $\bsbeta\in (0,\bstau)$.  Then
      \eqref{eq:assmulti} and Thm.~\ref{thm:hille} imply for all
      $z_j=x_j+\ii y_j\in S_{\tau_j}$, $j\in\{2,\dots,d\}$,
      \begin{equation}\label{eq:assmulti2}
        \left|\int_\R h_{\nu_1}(x_1)F(x_1,z_2,\dots,z_d) \dd x_1\right|
        \le 
        K(\beta_1) B(\bsbeta)\exp\left(-\sum_{j=2}^d|x_j|(\beta_j^2-y_j^2)^{1/2}\right) \exp(-\beta_1 (2\nu_1+1)^{1/2}).
      \end{equation}
      The function
      $(z_2,\dots,z_d)\mapsto\int_\R
      h_{\nu_1}(x_1)F(x_1,z_2,\dots,z_d) \dd x_1$ is well-defined and
      holomorphic (e.g., by the theorem in \cite{Mattner}) on
      $S_{\tau_2}\times\cdots\times S_{\tau_d}$.  Thus
      \eqref{eq:assmulti2} and Thm.~\ref{thm:hille} imply for all
      $z_j=x_j+\ii y_j\in S_{\tau_j}$, $j\in\{3,\dots,d\}$,
      \begin{align*}
        &\left|\int_\R h_{\nu_2}(x_2)\int_\R h_{\nu_1}(x_1) F(x_1,x_2,z_3,\dots,z_d)\dd x_1\dd x_2\right| 
          \nonumber\\
        &\qquad\le K(\beta_1) K(\beta_2) B(\bsbeta)\exp\left(-\sum_{j=3}^d|x_j|(\beta_j^2-y_j^2)^{1/2}\right)
          \exp(-\beta_1 (2\nu_1+1)^{1/2}-\beta_2(2\nu_2+1)^{1/2}).
      \end{align*}
      Repeating the argument another $d-2$ times concludes the proof.
    \end{proof}

    Our goal is to bound the Fourier coefficients w.r.t.\ the
    orthonormal Hermite polynomials $(H_\bsnu)_{\bsnu\in\N_0^d}$ in
    $L^2(\R^d,\gamma_d)$.  Thm.~\ref{thm:hille} and
    Cor.~\ref{cor:hillemulti} instead provide bounds on the Fourier
    coefficients w.r.t.\ the Hermite functions
    $(h_\bsnu)_{\bsnu\in\N_0^d}$ in $L^2(\R)$.  The following
    multivariate version of Lemma \ref{lemma:iso1} relates the two.
    \begin{lemma}\label{lemma:iso2}
      Let $d\in\N$ and set
      \begin{equation*}
        \Theta:
        L^2(\R^d,\gamma_d) \to L^2(\R^d):
        f(\bsx) \mapsto f(2^{1/2}\bsx)\frac{\exp{(-\frac{\norm[2]{\bsx}^2}{2})}}{\pi^{\frac{d}{4}}}.
      \end{equation*}
      Then $\Theta$ is an isometric isomorphism and
      $\Theta(H_\bsnu)=h_\bsnu$ for all $\bsnu\in\N_0^d$. In particular,
      for every $f\in L^2(\R^d,\gamma_d)$
      \begin{equation}\label{eq:hnuHnu}
        \dup{f}{H_\bsnu}_{L^2(\R^d,\gamma_d)}
        =\dup{\Theta(f)}{h_\bsnu}_{L^2(\R^d)}\qquad\forall\,\bsnu\in\N_0^d.
      \end{equation}
    \end{lemma}

    Equation \eqref{eq:hnuHnu} shows that, as long as $\Theta(f)$
    satisfies the Assumptions of Cor.~\ref{cor:hillemulti}, we have a
    bound of the type \eqref{eq:fnbound} on the Hermite coefficients
    $|\dup{f}{H_\bsnu}_{L^2(\R,\gamma_1)}|$.  Upon observing that
    $\Theta(f)$ satisfies Assumption \ref{ass:analytic} if $f$
    satisfies the assumptions of Cor.~\ref{cor:hillemulti},
    A version of this theorem has already been shown
    with essentially the same argument in \cite[Lemma 4.6,~Thm.~4.1]{MR2318799}. 
    For completeness and because our
    statement and assumptions slightly differ\footnote{In particular
      we allow for stronger growth of $f$ as $x \to \pm\infty$.} 
    from \cite{MR2318799}, we provide the proof in the appendix.
      
    \begin{theorem}\label{thm:finite}
      Let $f:\R^d\to\R$ satisfy Assumption \ref{ass:analytic} for some
      $\bstau\in (0,\infty)^d$.

      Then for all $\bsbeta\in (0,\bstau)$ exist $C> 0$
      (depending on $\bsbeta$) such that there holds,
      with $K(\beta_j)$ as in Thm.~\ref{thm:hille},
      \begin{enumerate}
      \item for all $\bsnu\in\N_0^d$
        \begin{equation} \label{eq:CoefBd}
          |\dup{f}{H_\bsnu}_{L^2(\R^d,\gamma_d)}| 
       \le \pi^{-d/4} B(\bsbeta) \left(\prod_{j=1}^d K(\beta_j)\right)
          \exp\left(-\sum_{j=1}^d\beta_j (2\nu_j+1)^{1/2}\right),
        \end{equation}
      \item for all $\eps\in (0,1)$ with
        \begin{equation}\label{eq:Leps}
          \Lambda_\eps:=\setc{\bsnu\in\N_0^d}{\nu_j \le \left(\frac{\log(\eps)}{\beta_j}\right)^2}
        \end{equation}
        and
        \begin{equation}\label{eq:delta}
          \delta(\bsbeta):=\left(\prod_{j=1}^d\beta_j\right)^{\frac 1 d}
        \end{equation}
        holds
        \begin{equation}\label{eq:f-Hermf}
          \normc[L^2(\R^d,\gamma_d)]{f-\sum_{\bsnu\in\Lambda_\eps}\dup{f}{H_\bsnu}H_\bsnu}
          \le C\eps
          \le C \exp\left(-2^{-\frac 1 2}\delta(\bsbeta)|\Lambda_\eps|^{\frac{1}{2d}}\right).
        \end{equation}
      \end{enumerate}
    \end{theorem}

    \subsection{$\ReLU$ neural network
      approximation}\label{sec:finReLU}
    The polynomial approximation result in the previous subsection
    together with the $\ReLU$ approximation result of Hermite
    polynomials provided in Sec.~\ref{sec:HermReLU} yield exponential
    convergence in $L^2(\R^d,\gamma_d)$ of DNN approximations with
    $\ReLU$ activations.  We prepare the proof of the theorem by
    showing two basic properties of $\Lambda_\eps$.
    \begin{lemma}\label{lemma:epsconstraint}
      Fix $\bsbeta\in (0,\infty)^d$ and let $\Lambda_\eps$ be as in
      \eqref{eq:Leps}.  Then for all
      \begin{equation}\label{eq:epsconstraint}
        \eps\in \left(0, \exp\left(-\max_{j\in\{1,\dots,d\}} \beta_j\right)\right)
      \end{equation}
      holds
      $|\Lambda_\eps|\le 2^d
      \frac{\log(\eps)^{2d}}{\prod_{j=1}^d\beta_j^2}$.  Furthermore,
      $m(\Lambda_\eps)\le\log(\eps)^2(\sum_{j=1}^d\beta_j^{-2})$.
    \end{lemma}
    \begin{proof}
      Due to \eqref{eq:epsconstraint} we have
      $(\log(\eps)/\beta_j)^2\ge 1$ and therefore
      \begin{equation*}
        |\Lambda_\eps| 
        \le 
        \prod_{j=1}^d\left(1+\left(\frac{\log(\eps)}{\beta_j}\right)^2\right)
        \le 
        2^d \frac{\log(\eps)^{2d}}{\prod_{j=1}^d\beta_j^2}.
      \end{equation*}
      In addition
      \begin{equation*}
        m(\Lambda_\eps) =\max_{\bsnu\in\Lambda_\eps}|\bsnu|_1
        \le
        \sum_{j=1}^d \frac{\log(\eps)^2}{\beta_j^2}.\qedhere
      \end{equation*}
    \end{proof}
    
    \begin{theorem}\label{thm:finite_relu}
      Let $f:\R^d\to\R$ satisfy Assumption \ref{ass:analytic} for some
      $\bstau\in (0,\infty)^d$.

      Then for all $\bsbeta\in (0,\bstau)$ exists $C>0$ (depending on
      $\bsbeta$ and $d$) such that for all $N\in\N$ exists a $\ReLU$
      network $\tilde f_N:\R^d\to\R$ such that
      \begin{equation}\label{eq:sizedepthfN}
        \size(\tilde f_N)\le CN(1+\log(N)),\qquad
        \depth(\tilde f_N)\le C N^{\frac{3}{2d+7}}(1+\log(N))^2,
      \end{equation}
      and
      \begin{equation}\label{eq:errfN}
        \norm[L^2(\R^d,\gamma_d)]{f-\tilde f_N}
        \le C \exp\left(-2^{-\frac 1 2}\delta(\bsbeta) N^{\frac{1}{2d+7}}\right).
      \end{equation}
    \end{theorem}
    \begin{proof}
      For $M>0$, define with $\delta(\bsbeta)$ as in \eqref{eq:delta}
      \begin{equation}\label{eq:epsM}
        \eps_M:=
        \exp\left(-\left(\frac{M\prod_{j=1}^d\beta_j^2}{2^d} \right)^{\frac{1}{2d}}\right)
               =
               \exp\left(- 2^{-\frac12} \delta(\bsbeta) M^{\frac{1}{2d}}\right)\;.
      \end{equation}
      Assume $M>2$ is chosen so large that
      $\eps_M < \exp\left(-\max_{j\in\{1,\dots,d\}} \beta_j\right)$,
      i.e.\ \eqref{eq:epsconstraint} holds.

      By Lemma \ref{lemma:epsconstraint} it holds
      $|\Lambda_{\eps_M}|\le M$.  Furthermore, \eqref{eq:f-Hermf} with
      $\eps_M$ in place of $\eps$ implies
     \begin{equation}\label{eq:errfinite0}
        \normc[L^2(\R^d,\gamma_d)]{f-\sum_{\bsnu\in\Lambda_{\eps_M}}
          \dup{f}{H_\bsnu}H_\bsnu} 
        \le
         C \eps_M \le C \exp(-2^{-\frac{1}{2}} \delta(\bsbeta) M^{\frac{1}{2d}}).
      \end{equation}
      Next, 
      let $(\tilde H_{\varepsilon_M,\bsnu})_{\bsnu\in\Lambda_{\eps_M}}$
      be the $\ReLU$ approximation from Thm.~\ref{thm:reluhermite}.
      As the coefficients $\dup{f}{H_\bsnu}$ are summable according to \eqref{eq:CoefBd}, 
      we get
      \begin{equation*}
        \normc[L^2(\R^d,\gamma_d)]{\sum_{\bsnu\in\Lambda_{\eps_M}}\dup{f}{H_\bsnu}|H_\bsnu - \tilde H_{\varepsilon_M,\bsnu}|}
        \le %
        \eps_M
        \sum_{\bsnu\in\N_0^d}|\dup{f}{H_\bsnu}| 
        \le C \exp(-{2^{-\frac 1 2}}\delta(\bsbeta)%
        M^{\frac{1}{2d}}).
      \end{equation*}
      Together with \eqref{eq:errfinite0} we observe that the network
      \begin{equation*}
        \tilde g_M:=\sum_{\bsnu\in\Lambda_{\eps_M}}\dup{f}{H_\bsnu} \tilde H_{\varepsilon_M,\bsnu}
      \end{equation*}
      satisfies the error bound
      \begin{equation}\label{eq:errgm}
        \norm[L^2(\R^d,\gamma_d)]{f-\tilde g_M}
        \le C \exp\left(-2^{-\frac 1 2}\delta(\bsbeta) M^{\frac{1}{2d}}\right).        
      \end{equation}
      
      Next we bound the size and depth of $\tilde g_M$.  
      By Lemma \ref{lemma:epsconstraint} and \eqref{eq:epsM}
      \begin{equation*}
        m(\Lambda_{\eps_M})\le \log(\eps_M)^2\sum_{j=1}^d\beta_j^{-2}
        \le 
        M^{\frac 1 d}\delta(\bsbeta)^2 \sum_{j=1}^d\beta_j^{-2}
        \le C M^{\frac{1}{d}}
      \end{equation*}
      for some $C=C(d,\bsbeta)$. 
      It holds
      $\size(\tilde g_M)\le C|\Lambda_{\eps_M}|+\size{(\tilde H_{\varepsilon_M,\bsnu})_{\bsnu\in\Lambda_{\eps_M}}}$. 
      By Thm.~\ref{thm:reluhermite}
      \begin{align}\label{eq:sizegm}
        {\size}(\tilde g_M)
      &\le C|\Lambda_{\eps_M}|+C |\Lambda_{\eps_M}|m(\Lambda_{\eps_M})^3
        \log(1+m(\Lambda_{\eps_M}))d(\Lambda_{\eps_M})^2|\log(\varepsilon_M)| \nonumber
        \\
                           &\le C M+C M M^{\frac{3}{d}}%
                             \log(C M) d^2
                             M^{\frac{1}{2d}}\nonumber\\
                           &\le
                             C(1+M)^{1+\frac{3}{d}+\frac{1}{2d}}(1+\log(M)),
      \end{align}
      where $C$ depends on $\bsbeta$ and $d$ and may change its value
      after each inequality in the above computation.  
      Similarly, using again Thm.~\ref{thm:reluhermite},
      \begin{align}\label{eq:depthgm}
        {\depth}(\tilde g_M)
        &\le C+C m(\Lambda_{\eps_M})\log(1+m(\Lambda_{\eps_M}))^2d(\Lambda_{\eps_M})\log(1+d(\Lambda_{\eps_M}))|\log(\varepsilon_M)|
        \nonumber\\
        &\le C+CM^{\frac{1}{d}}\log(1+M)^2 d \log(1+d) M^{\frac{1}{2d}}
         \nonumber\\
        &\le CM^{\frac{1}{d}+\frac{1}{2d}}(1+\log(M))^2.
      \end{align}
      Setting $\tilde f_N:=\tilde g_M$ with
      $M:=N^{\frac{2d}{2d+7}}-1$, \eqref{eq:errgm}, \eqref{eq:sizegm},
      \eqref{eq:depthgm} imply the error, size and depth bounds
      \eqref{eq:sizedepthfN} and \eqref{eq:errfN}.

      Finally, the condition
      $\eps_M\le \exp\left(-\max_{j\in\{1,\dots,d\}} \beta_j\right)$
      corresponds to $M\ge M_0$ for some fixed $M_0$ depending on $d$
      and $\bsbeta$. Since the theorem holds for all $M\ge M_0$, it
      remains true for all $M\in\N$ after possibly adjusting the
      constant $C$.
    \end{proof}
    \section{%
      DNN approximation of infinite-parametric, analytic functions %
      in $L^2(\R^\N,\gamma)$}
    \label{sec:bxdXHol}
    In this section we consider the $\ReLU$-NN approximation of
    certain \emph{countably-parametric}, analytic maps from $U=\R^\N$
    to $\R$ in $L^2(\R^\N,\gamma)$.  Such maps arise as solutions of
    operator equations with Gaussian random field inputs, which are
    represented in an affine-parametric fashion, via a Parseval frame
    \cite{LPFrame09} such as e.g.\ a Karhunen-Lo\`eve or a
    L\'evy-Cieselskii expansion of the GRF.  We discuss an example in
    Sec.~\ref{sec:PDESol}.  The proof of NN approximation bounds
    proceeds in two stages. First, a polynomial chaos approximation is
    constructed based on the results in \cite{DNSZ20}, and second,
    this approximation is emulated by a deep $\ReLU$-NN using our
    results from the preceding sections.

    \subsection{Wiener polynomial chaos approximation}
    \label{sec:WPCAppr}
    We recall the notion of ($\bsb,\xi,\delta$)-holomorphy from
    \cite[Def.~6.1]{DNSZ20}.
    
    \begin{definition}[($\bsb,\xi,\delta$)-Holomorphy]
      \label{def:bdXHol}
      Let $\bsb=(b_j)_{j\in\N} \in (0,\infty)^\N$ and let $\xi>0$,
      $\delta>0$.

      We say that $\bsvarrho\in (0,\infty)^N$ is
      \textbf{$(\bsb,\xi)$-admissible} if for every $N\in\N$
      \begin{equation}\label{eq:adm}
        \sum_{j=1}^N b_j\varrho_j\leq \xi\,.
      \end{equation}
  
      A real-valued function $u\in L^2(U,\gamma)$ is called
      \textbf{$(\bsb,\xi,\delta)$-holomorphic} if
      \begin{enumerate}
      \item\label{item:hol} for every finite $N\in\N$ there exists
        $u^N:\R^N\to \R$, which, for every $(\bsb,\xi)$-admissible
        $\bsvarrho\in (0,\infty)^N$, admits a holomorphic extension
        (denoted again by $u^N$) from $\cS_\bsvarrho\to \C$; moreover
        for all $N<M$
        \begin{equation}\label{eq:un=um}
          u^N(y_1,\dots,y_N) 
          =
          u^M(y_1,\dots,y_{N},0,\dots,0)\qquad\forall (y_j)_{j=1}^N\in\R^N,
        \end{equation}

      \item\label{item:varphi} for every $N\in\N$ there exists
        $\varphi_N:\R^N\to\R_+$ such that
        $\norm[L^2(\R^N,\gamma_N)]{\varphi_N}\le\delta$ and
        \begin{equation} \label{ineq[phi]}
          \sup_{\bsvarrho\in(0,\infty)^N\text{ is
            }(\bsb,\xi)-\text{adm.}}~\sup_{\bsz\in
            \cB(\bsvarrho)} |u^N(\bsy+\bsz)|\le
          \varphi_N(\bsy)\qquad\forall\bsy\in\R^N,
        \end{equation}
      \item\label{item:vN} with $\hat u^N:U\to \R$ defined by
        $\hat u^N(\bsy) :=u^N(y_1,\dots,y_N)$ for $\bsy\in U$ it
        holds
        \begin{equation}
          \lim_{N\to\infty}\norm[L^2(U,\gamma)]{u-\hat u^N}=0.
        \end{equation}
      \end{enumerate}
    \end{definition}

    In the following, for $u:U\to\R$ as in Def.~\ref{def:bdXHol}, we set
    \begin{equation*}
      u_\bsnu:=\int_U u(\bsy)H_\bsnu(\bsy) \dd\gamma(\bsy)\in\R, \quad \bsnu\in \cF,
    \end{equation*}
    which are the so-called Wiener-Hermite polynomial chaos (PC) expansion coefficients.
    They are well-defined since $u\in L^2(U,\gamma)$ and
    $H_\bsnu\in L^2(U,\gamma)$, and thus $uH_\bsnu\in L^1(U,\gamma)$.

    The following theorem specifies Hermite PC coefficient
    summability, %
    see \cite[Cor.~7.9]{DNSZ20}.
    \begin{theorem}\label{thm:be1}
      Let $u$ be $(\bsb,\xi,\delta)$-holomorphic for some
      $\bsb\in \ell^p(\N)$ and some $p\in (0,\frac{2}{3})$. 
      Then
      $(u_\bsnu)_{\bsnu\in\cF}\in \ell^{2p/(2-p)}(\cF)$.
    \end{theorem}

    Since $p\in (0,\frac{2}{3})$, Thm.~\ref{thm:be1} implies
    $(|u_\bsnu|)_{\bsnu\in\CF}\in \ell^1(\cF)\hookrightarrow \ell^2(\cF)$. 
    Since
    $(H_\bsnu)_{\bsnu\in\cF}$ is an orthonormal basis of
    $L^2(U,\gamma)$, the expansion
    \begin{equation}\label{eq:uexp}
      u=\sum_{\bsnu\in\cF}u_\bsnu H_\bsnu,
    \end{equation}
    converges in $L^2(U,\gamma)$. 
    Truncating this
    expansion yields an approximation to $u$.  Proving convergence
    rates of $N$-term truncated Wiener-Hermite pc expansions requires
    a more specific result however. It is given in the next theorem
    that is shown in \cite[Thm.~7.8, Lemmata 9.5 and 9.6]{DNSZ20}.

\begin{theorem}\label{thm:be2}
  Let $u$ be $(\bsb,\xi,\delta)$-holomorphic for some
  $\bsb\in \ell^p(\N)$ and some $p\in (0,\frac{2}{3})$. Let $r>2/p-1$.

  Then there exists $K>0$ such that with
  \begin{equation}\label{eq:cnu}
    c_{\bsnu}:= \prod_{j\in\supp\bsnu}\max\left\{1,K
      b_j^{p-1}\right\}^{2} \nu_j^{r}\qquad\bsnu\in\cF,
  \end{equation}
  it holds
  \begin{enumerate}
  \item $(c_{\bsnu}^{-1})_{\bsnu\in\cF}\in\ell^{\frac{p}{2(1-p)}}$,
  \item $\sum_{\bsnu\in\cF} c_\bsnu u_\bsnu^2<\infty$.
  \end{enumerate}
\end{theorem}

In the following, for $c_\bsnu$ as in \eqref{eq:cnu}
  we let similar to \eqref{eq:Leps} for $\eps>0$
  \begin{equation}\label{eq:Leps2}
    \Lambda_\eps :=\set{\bsnu\in\cF}{c_\bsnu^{-1}\ge\eps}.
  \end{equation}
  It is easy to see that the definition of $c_\bsnu$ in \eqref{eq:cnu}
  implies $\Lambda_\eps$ to be finite and downward closed.
\begin{corollary}\label{cor:conv}
  Consider the setting of Thm.~\ref{thm:be2}. Then for every
  $\eps>0$
  \begin{equation*}
    \normc[L^2(U,\gamma)]{u-\sum_{\bsnu\in\Lambda_\eps}u_\bsnu H_\bsnu}
    \le \eps^{1/2} \left(\sum_{\bsnu\in\cF}c_\bsnu
      u_\bsnu^2\right)^{1/2}.
  \end{equation*}
  In addition,
  \begin{equation}\label{eq:epsleLeps}
    \eps \le \norm[\ell^{p/(2(1-p))}]{(c_\bsnu^{-1})_{\bsnu\in\cF}}
    |\Lambda_\eps|^{-\frac{2(1-p)}{p}}
  \end{equation}
  so that in particular
  with the
  finite constant
  $C:=(\norm[\ell^{p/(2(1-p))}]{(c_\bsnu^{-1})_{\bsnu\in\cF}}\sum_{\bsnu\in\cF}c_\bsnu
  u_\bsnu^2)^{1/2}$ holds
  $\norm[L^2(U,\gamma)]{u-\sum_{\bsnu\in\Lambda_\eps}u_\bsnu
    H_\bsnu}\le C |\Lambda_\eps|^{-\frac{1}{p}+1}$.
\end{corollary}
\begin{proof}
  By \eqref{eq:uexp} and the orthogonality of the
  $(H_\bsnu)_{\bsnu\in\CF}$ in $L^2(U,\gamma)$,
  \begin{equation*}
    \normc[L^2(U,\gamma)]{u-\sum_{\bsnu\in\Lambda_\eps}u_\bsnu H_\bsnu}
    =\left(\sum_{\bsnu\in\CF\backslash\Lambda_\eps} u_\bsnu^2\right)^{1/2}.
  \end{equation*}
  It holds
  \begin{equation*}
    \sum_{\bsnu\in\CF\backslash\Lambda_\eps} u_\bsnu^2=
    \sum_{\bsnu\in\CF\backslash\Lambda_\eps} u_\bsnu^2c_\bsnu c_\bsnu^{-1}
    \le \sum_{\bsnu\in\cF}
    u_\bsnu^2c_\bsnu \sup_{\bsnu\in\cF\backslash \Lambda_\eps}c_\bsnu^{-1}
    \le \eps \sum_{\bsnu\in\cF} u_\bsnu^2c_\bsnu
  \end{equation*}
  by definition of $\Lambda_\eps=\set{\bsnu}{c_\bsnu^{-1}\ge\eps}$.

  By Thm.~\ref{thm:be2} we have
  $\sum_{\bsnu\in\cF} u_\bsnu^2c_\bsnu<\infty$ and
  $(c_\bsnu^{-1})_{\bsnu\in\cF}\in\ell^{p/(2(1-p))}$.
  Hence, using $c_\bsnu \le \eps^{-1}$ for all $\bsnu\in\Lambda_\eps$,
  \begin{equation*}
    |\Lambda_\eps| = \sum_{\bsnu\in\Lambda_\eps} c_\bsnu^{-\frac{p}{2(1-p)}}
    c_\bsnu^{\frac{p}{2(1-p)}}
    \le \eps^{-\frac{p}{2(1-p)}}\sum_{\bsnu\in\cF}c_\bsnu^{-\frac{p}{2(1-p)}}
  \end{equation*}
  and consequently
  \begin{equation*}
    \eps \le |\Lambda_\eps|^{-\frac{2(1-p)}{p}}
    \norm[\ell^{p/(2(1-p))}]{(c_\bsnu^{-1})_{\bsnu\in\cF}}.
  \end{equation*}
  In all
  \begin{equation*}
    \normc[L^2(U,\gamma)]{u-\sum_{\bsnu\in\Lambda_\eps}u_\bsnu H_\bsnu}
    \le \eps^{1/2}
    \left(\sum_{\bsnu\in\cF}c_\bsnu
      u_\bsnu^2\right)^{1/2}
    \le |\Lambda_\eps|^{-\frac{1}{p}+1}
    \left(\norm[\ell^{p/(2(1-p))}]{(c_\bsnu^{-1})_{\bsnu\in\cF}}\sum_{\bsnu\in\cF}c_\bsnu
      u_\bsnu^2\right)^{1/2}
  \end{equation*}
  as claimed.
\end{proof}
\subsection{$\ReLU$ neural network approximation}
\label{sec:ReLUWPC}
Let again $(c_{\bsnu})_{\bsnu\in\cF}$ be as in \eqref{eq:cnu} with
some $\bsb\in \ell^p(\N)$ and $K>0$, and let $\Lambda_\eps$ be as in \eqref{eq:Leps2}. 
As in \cite{ZS17}, we investigate %
the quantities $m(\Lambda_\eps)$ and $d(\Lambda_\eps)$ defined in \eqref{eq:md}, 
as $\eps\searrow 0$.
\begin{lemma}\label{lemma:dm}
  Assume that there exists $C_0>0$, $s>0$ and $p>0$ such that
  $\bsb = (b_j)_{j\in\N}\in\ell^p(\N)$ and
  $b_j\ge C_0 j^{-\frac{s}{2(1-p)}}$ for all $j\in\N$.  Let
  $(c_{\bsnu})_{\bsnu\in\cF}$ be as in \eqref{eq:cnu} for this $\bsb$
  and some $K>0$. Then
  \begin{equation}
    d(\Lambda_\eps) = o(\log(|\Lambda_\eps|))\quad\text{and}\quad
    m(\Lambda_\eps) =
    O\left(|\Lambda_\eps|^{\frac{s}{r}}\right)\qquad\text{as }\eps\searrow 0.
  \end{equation}
\end{lemma}
\begin{proof}
  With $K$ as in \eqref{eq:cnu} set
  \begin{equation*}
    \hat \varrho_j:=\max\{1,K b_j^{p-1}\}.
  \end{equation*}
  Throughout this proof we assume wlog that $(b_j)_{j\in\N}$ is
  monotonically decreasing (otherwise permute the sequence
  $(b_j)_{j\in\N}$ accordingly).
  
  Denote by $(x_j)_{j\in\N}$ a monotonically decreasing rearrangement
  of $(c_{\bsnu}^{-1})_{\bsnu\in\cF}$. Since $b_j^{p-1}\to \infty$,
  there exists $C_1>0$ such that $\hat \varrho_j\le C_1 b_j^{p-1}$ for
  all $j$. We have
  $c_{\bse_j}^{-1} = \hat \varrho_j^{-2}\ge C_1^{-2} b_j^{-2(p-1)}
  =C_1^{-2}b_j^{2(1-p)}$. Since $b_j$ is monotonically decreasing, by
  definition of $x_j$ it must hold $x_j\ge C_1^{-2}b_j^{2(1-p)}$.
  With the assumption $b_j\ge C_0j^{-\frac{s}{2(1-p)}}$ we get
  \begin{equation}\label{eq:xjge}
    x_j\ge C_1^{-2}C_0
    j^{-s}=C_2j^{-s}.
  \end{equation}
  We will show that there are fixed constants $C_3$, $C_4$, $C_5>0$
  depending on $(\hat\varrho_j)_{j\in\N}$ but independent of $d\in\N$
  so that there holds
  \begin{equation}\label{eq:toshow}
    \max_{\set{\bsnu\in\cF}{|\bsnu|_0=d}}c_{\bsnu}^{-1} \le C_3 d^{-C_4 d}\qquad
    \text{and} \qquad 
    \max_{\set{\bsnu\in\cF}{|\bsnu|=m}}c_{\bsnu}^{-1}\le C_5 m^{-{r}}
    \;.
  \end{equation}
  Denote $F(t):=C_3 t^{-C_4 t}$.  Then $F:[1,\infty)\to (0,C_3]$ is
  strictly monotonically decreasing and bijective.  Hence
  $F^{-1}:(0,C_3]\to [1,\infty)$ is strictly decreasing and bijective.
  Using \eqref{eq:toshow} and \eqref{eq:xjge} it
  holds %
  \begin{equation*}
    F(d(\Lambda_N))
    \ge
    \max_{\set{\bsnu\in\cF}{|\bsnu|_0=d(\Lambda_N)}}c_{\bsnu}^{-1}\ge
    \max_{\set{\bsnu\in\Lambda_N}{|\bsnu|_0=d(\Lambda_N)}}c_{\bsnu}^{-1}
    \ge \min_{\bsnu\in\Lambda_N}c_{\bsnu}^{-1}
    =x_N
    \ge C_2 N^{-s}. 
  \end{equation*}
  If $N$ is so large that $C_2N^{-s}\le C_3$, we may apply $F^{-1}$ on
  both sides and conclude that $d(\Lambda_N)\le F^{-1}(C_2
  N^{-s})$. Since $F^{-1}(t)=o(-\log(t))$ as $t\to 0$, we obtain
  \begin{equation*}
    d(\Lambda_N)=o(-\log(N^{-s}))=o(\log(N))
  \end{equation*}
  as $N\to\infty$.  Similarly, letting $G(t):= C_5 t^{-{r}}$ and
  observing that $G^{-1}(t)=(t/C_5)^{-1/{r}}=O(t^{-1/{r}})$ as
  $t\to 0$, one shows that
  \begin{equation*}
    m(\Lambda_N)\le G^{-1}(C_2N^{-s}) =O(N^{\frac{s}{r}})
  \end{equation*}
  as $N\to\infty$.

  It remains to verify \eqref{eq:toshow}. Without loss of generality
  we assume $(\hat\varrho_j^{-1})_{j\in\N}$ to be monotonically
  decreasing. Using H\"older's inequality and the fact that
  $(\hat\varrho_j^{-1})_{j\in\N}\in\ell^q(\N)$ with $q:=p/(1-p)$
  (since $\hat\varrho_j^{-1}\sim b_j^{1-p}$ and
  $(b_j)_{j\in\N}\in\ell^p(\N)$) one can show that
  $\hat\varrho_j^{-1}\le
  \norm[\ell^q(\N)]{(\hat\varrho_j^{-1})_{j\in\N}} j^{-1/q}$ for all
  $j\in\N$ (see for example \cite[Lemma 2.9]{ZS17}).  Therefore with
  $C_6:=\norm[\ell^q(\N)]{(\hat\varrho_j^{-1})_{j\in\N}}$
  \begin{align*}
    \max_{\set{\bsnu\in\cF}{|\bsnu|_0=d}}
    c_{\bsnu}^{-1} = \prod_{j=1}^d \hat\varrho_j^{-2}\le 
    \prod_{j=1}^d (C_6j^{-\frac{1}{q}})^{2}\le C_6^{2d} (d!)^{-\frac{2}{q}}
    \le C_6^{2d}
    (\e^{-d} d^d)^{-\frac{2}{q}} \le C_6^{2d}
    \e^{d\frac{2}{q}} d^{-d \frac{2}{q}}.
  \end{align*}
  This implies the first inequality in \eqref{eq:toshow}.
  To show the second inequality we note that
  \begin{equation}\label{eq:mest}
    \max_{|\bsnu|_1=m} c_\bsnu^{-1} =   \max_{|\bsnu|_1=m}
    \prod_{j\in\supp\bsnu} \hat\varrho_j^{-1}\nu_j^{-{r}}.
  \end{equation}
  Observe that for $\bsnu=\bsmu+\bse_i$
  \begin{equation*}
    \frac{c_\bsnu^{-1}}{c_\bsmu^{-1}} = 
    \frac{\prod_{j\in\supp\bsnu}
      \hat\varrho_j^{-1}\nu_j^{-{r}}}{\prod_{j\in\supp\bsnu}
      \hat\varrho_j^{-1}\mu_j^{-{r}}} = \begin{cases}
      \hat\varrho_i^{-1} &\text{if }\nu_i=0,\\
      \left(\frac{\nu_i}{\nu_{i}+1}\right)^{r} &\text{otherwise}.
    \end{cases}
  \end{equation*}
  By definition $\hat\varrho_j=\max\{1,K\varrho_j\}\ge 1$ for all $j$
  and thus $\hat\varrho_j^{-1}\le 1$ for all $j\in\N$. Now suppose
  that $J\in\N$ is so large that $\hat\varrho_j^{-1}\le 2^{-{r}}$ for
  all $j\ge J$. Then for all $m\ge J$, since $(n/(n+1))^{r}$ is
  monotonically increasing as a function of $n\in\N$,
  \begin{equation*}
    \max_{|\bsnu|_1=m}
    \prod_{j\in\supp\bsnu} \hat\varrho_j^{-1}\nu_j^{-r}\le
    \prod_{n=1}^{m-J} \left(\frac{n}{n+1}\right)^{r} = (m-J+1)^{-r}.
  \end{equation*}
  Together with \eqref{eq:mest} this implies the second inequality in
  \eqref{eq:toshow}.
\end{proof}
We are now in position to state our main result in this section.
  It provides $\ReLU$-NN expression rates for countably-parametric,
  $(\bsb,\xi,\delta)$-holomorphic maps.
\begin{theorem}\label{thm:DNNGRF}
  Let $u:U\to\R$ be $(\bsb,\xi,\delta)$-holomorphic for some
  $\bsb\in\ell^p(\N)$ with a $p\in (0,\frac{2}{3})$. Fix $\delta>0$ arbitrarily
  small.

Then there exists a constant $C>0$ (depending on $u$) such that for
every $N\in\N$ there exists a $\ReLU$-NN $\tilde u_N$ with
  \begin{equation}\label{eq:tunerr}
    \norm[L^2(U,\gamma)]{u(\bsy)-\tilde u_N(\bsy)}
    \le C N^{-\frac{1}{p}+1},
  \end{equation}
  and it holds
  \begin{equation}\label{eq:SizDep}
    \size(\tilde u_N)\le C N^{1+\delta},\qquad
    \depth(\tilde u_N)\le C N^\delta.
  \end{equation}
\end{theorem}
\begin{proof}
  Define $\hat b_j:=\max\{b_j,j^{-2/p}\}$. Then
  $\hat\bsb\in\ell^{p}(\N)$ and $\hat b_j\ge b_j$ for all $j\in\N$.
  The definition of $(\bsb,\xi,\delta)$-holomorphy implies that $u$ is
  also $(\hat\bsb,\xi,\delta)$-holomorphic. As in \eqref{eq:Leps2}
  we let $\Lambda_\eps=\set{\bsnu}{c_\bsnu^{-1}\ge\eps}$, 
  with
  $c_\bsnu$ as in \eqref{eq:cnu} defined with $\hat\bsb$ in place of
  $\bsb$. 
  We fix $r>2/p-1 > 1$ in \eqref{eq:cnu} 
  large enough such that with  
  $s:=(2/p)/(2(1-p))>0$ it holds $\frac{4s}{r}<\delta$.

  For $\eps\in (0,1]$ set
    $u_\eps:=\sum_{\bsnu\in\Lambda_\eps}u_\bsnu H_\bsnu$. 
    By Cor.~\ref{cor:conv} with
    $C_1:=(\sum_{\bsnu\in\cF}c_\bsnu u_\bsnu^2)^{1/2}<\infty$ holds
    $\norm[L^2(U,\gamma)]{u-u_\eps}\le C_1 \eps^{1/2}$.
    By Thm.~\ref{thm:reluhermite}, 
    there exists a $\ReLU$-NN
    $\Phi=(\tilde H_{\eps,\bsnu})_{\bsnu\in\Lambda_\eps}:
    \R^{|\supp\Lambda_\eps|}\to \R^{|\Lambda_\eps|}$ such that
    $\norm[L^2(U,\gamma)]{H_\bsnu-\tilde H_{\eps,\bsnu}}\le \eps$ 
    for each $\bsnu\in\Lambda_\eps$.
    Then the NN
    $\tilde u_\eps:=\sum_{\bsnu\in\Lambda_\eps}u_\bsnu\tilde
    H_{\eps,\bsnu}$ satisfies
    \begin{equation*}
      \norm[L^2(U,\gamma)]{u_\eps-\tilde u_\eps}
      \le
      \sum_{\bsnu\in\Lambda_\eps}|u_\bsnu|\eps
      = 
      \eps\sum_{\bsnu\in\Lambda_\eps}|u_\bsnu|.
    \end{equation*}
    By Thm.~\ref{thm:be1} it holds
    $C_2:=\sum_{\bsnu\in\Lambda_\eps}|u_\bsnu|<\infty$. 
    Hence (using $0 < \eps \le \eps^{1/2}\leq 1$)
    \begin{equation}\label{eq:u-tueps}
      \norm[L^2(U,\gamma)]{u-\tilde u_\eps}
      \le
      \norm[L^2(U,\gamma)]{u-u_\eps}+\norm[L^2(U,\gamma)]{u_\eps-\tilde u_\eps}
      \le 
      (C_1+C_2)\eps^{1/2}.
    \end{equation}
    Next, by Lemma \ref{lemma:dm}
    \begin{equation*}
      d(\Lambda_\eps)\le C\log(|\Lambda_\eps|),\qquad
      m(\Lambda_\eps)\le C |\Lambda_\eps|^{\frac{s}{r}},
    \end{equation*}
    where $s=(2/p)/(2(1-p))>0$. 
    Thm.~\ref{thm:reluhermite} thus
    implies the bounds (here we use $4s/r<\delta$)
    \begin{equation*}
      \size(\tilde u_\eps)\le |\Lambda_\eps|+\size(\Phi)\le
      C|\Lambda_\eps| m(\Lambda_\eps)^4 d(\Lambda_\eps)^2\log(\eps^{-1})
      \le C|\Lambda_\eps| |\Lambda_\eps|^{\frac{4s}{r}} \log(|\Lambda_\eps|)^4
      \le C|\Lambda_\eps|^{1+\delta}.
    \end{equation*}
    Similarly
    \begin{equation*}
      \depth(\tilde u_\eps)\le 1+\depth(\Phi)\le
      1+Cm(\Lambda_\eps)^3d(\Lambda_\eps)^2\log(\eps^{-1})
      \le 1+C|\Lambda_\eps|^{\frac{3s}{r}}\log(|\Lambda_\eps|)^3
      \le 1+C|\Lambda_\eps|^\delta.
    \end{equation*}
    Now using \eqref{eq:epsleLeps} we have with
    $C_3:=\norm[\ell^{p/(2(1-p))}]{(c_\bsnu^{-1})_{\bsnu\in\cF}}$
    \begin{equation*}
      |\Lambda_\eps|\le \frac{1}{C_3}\eps^{-\frac{p}{2(1-p)}}.
    \end{equation*}
    Finally, for $N\in\N$ so large that $\eps_N:=(C_3 N)^{-2(1-p)/p}$
    is less or equal to $1$ we have in particular
    $|\Lambda_{\eps_N}|\le N$. This choice yields a network
    $\tilde u_N:=\tilde u_{\eps_N}$ satisfying the size and depth
    bounds \eqref{eq:SizDep}, as well as the error bound
    \eqref{eq:tunerr} due to \eqref{eq:u-tueps} 
    and the definition of $\eps_N$.
\end{proof}
\begin{remark}
  The network $\tilde u_N$ in Thm.~\ref{thm:DNNGRF} has size upper
  bounded by $CN$ but takes infinitely many inputs $\bsy\in\R^\N$.
  This is to be understood as follows: $\tilde u_N$ is a network with
  only finitely many inputs $(y_j)_{j\in S_N}$ for some finite set
  $S_N\subseteq\N$. All other inputs are ignored. If $b_j$ in
  Def.~\ref{def:bdXHol} is monotonically increasing, the proof shows
  that one can choose $S_N=(j)_{j\le CN}$.
\end{remark}
We remark that inspection of the proof actually reveals 
slightly more precise bounds on
$\size(\tilde u_\eps)$ and on $\depth(\tilde u_\eps)$ 
than the claim \eqref{eq:SizDep}.
\section{DNN Expression rate bounds for response-surfaces of \\PDEs with GRF input}
\label{sec:PDESol}
We illustrate the expression rate bounds for the infinite-parametric
case obtained in Sec.~\ref{sec:ReLUWPC}, by applying them to
  pushforwards of Gaussian measures under PDE solution maps. 
For
definiteness, we consider standard, linear elliptic second order
diffusion in a bounded Lipschitz domain $\dom\subseteq\R^d$.
For a given source term $f\in H^{-1}(\dom)$, and for a
log-Gaussian diffusion coefficient $a=\exp(g)$ with a GRF $g$ taking
values in $L^\infty(\dom)$, consider the Dirichlet problem
\begin{equation}\label{eq:DiffD}
  \nabla\cdot(a \nabla u) + f = 0\quad\mbox{in}\quad \dom
  \;,
  \quad 
  u|_{\partial \dom} = 0
  \;.
\end{equation}
We assume the log-Gaussian random field $g=\log(a)$ to admit a
representation in terms of a Karhunen-Lo\`eve expansion
\begin{equation}\label{eq:GRFg}
  \log(a(x,\bsy)) 
  = g(x,\bsy) 
  = \sum_{j\geq 1} y_j \psi_j(x) \;,\qquad x\in \dom, %
\end{equation}
where $\bsy = (y_j)_{j\in\N} \in \R^\N$ 
with the $y_j\in\R$ iid centered
standard Gaussian, and 
for certain $\psi_j\in L^\infty(\dom)$.

\begin{remark}
  The functions $g$ and $u$ in \eqref{eq:DiffD}-\eqref{eq:GRFg}
  are well-defined for instance
  under the following assumptions:
  Assume that for every $j\in \N$ %
$\psi_j \in L^\infty(\dom)$ and there exists 
$(\lambda_j)_{j\geq 1} \in [0,\infty)^{\N}$ such that
(i) $( \exp(-\lambda_j^2) )_{j\geq 1} \in \ell^1(\N)$ 
and
(ii) $ \sum_{j\in \N}\lambda_j |\psi_j| $ converges in $L^\infty(\dom)$.

Then the set $U_0 := \set{\bsy \in \R^\N}{g(\bsy) \in L^\infty(\dom)}$
has full measure, i.e.\ $\gamma(U_0) = 1$, and for every $k\in \N$
holds $\E(\exp(k\| g \|_{L^\infty(\dom)}) < \infty$ (see
\cite[Thm. 2.2]{BCDM}).  Furthermore, for all $f\in H^{-1}(\dom)$ and
for every $\bsy\in U_0$, \eqref{eq:DiffD} has a unique solution
$u(\bsy)\in H^1_0(\dom)$, and
$u\in L^k(U,\gamma;H^1_0(\dom))$
for all finite $k\in \N$.
\end{remark}

For an \emph{observable} $G\in H^{-1}(\dom)$, 
we consider the \emph{countably-parametric,
deterministic PDE response map} $G\circ u$
with $u$ denoting the solution to \eqref{eq:DiffD} 
for the log-Gaussian random field $a$ as in \eqref{eq:GRFg}. 
This map can be formally expressed as
\begin{equation}\label{eq:Uexp}
  G(u(\bsy)) = {\mathcal U}\left(\exp\left(\sum_{j\in \N} y_j\psi_j\right)\right),\qquad \bsy\in U,
\end{equation}
for some mapping $\mathcal{U}: L^\infty(\dom)\to \R$.
More precisely, $\mathcal{U}$ maps a
diffusion coefficient $a\in L^\infty(\dom)$ to the observable
$G$ applied to the solution of \eqref{eq:DiffD}. 
By the complex Lax-Milgram Lemma, 
the map $\mathcal{U}$ is in particular well-defined
on the set $\set{a\in L^\infty(\dom,\C)}{\essinf_{x\in\dom}\Re(a)>0}$.

An abstract result shown in \cite[Lemma 7.10]{DNSZ20}, implies that
functions of the type
$\bsy\mapsto \mathcal{U}(\exp(\sum_j y_j\psi_j))$ as in
\eqref{eq:Uexp} are $(\bsb,\xi,\delta)$-holomorphic with
$b_j=\norm[L^\infty(\dom)]{\psi_j}$, as long as $\mathcal{U}$ is a
holomorphic map between two Banach spaces and it holds
$\bsb\in\ell^p(\N)$ for some $p\in (0,\frac{2}{3})$. 
More precisely, \cite[Lemma 7.10]{DNSZ20} shows that 
(under certain additional assumptions) 
the functions
\begin{equation*}
  v^N(y_1,\dots,y_N):=
  {\mathcal U}\left(\exp\left(\sum_{j=1}^N y_j\psi_j\right)\right)
\end{equation*}
converge towards some $v\in L^2(U,\gamma)$ as $N\to\infty$, and this
$v$ is $(\bsb,\xi,\delta)$-holomorphic. In this sense
$v(\bsy)=G(u(\bsy))\in L^2(U,\gamma)$ is well-defined. We emphasize
that the crucial assumption of $\mathcal{U}$ being holomorphic can be
shown for the diffusion problem \eqref{eq:DiffD}, but the result is
far from limited to this specific PDE: 
similar statements can be shown
for instance for the Maxwell's equations \cite{Jerez-Hanckes:2017aa}
or for well-posed parabolic PDEs \cite{DNSZ20} 
(see \cite[Section 7]{DNSZ20}, where well-definedness and
$(\bsb,\xi,\delta)$-holomorphy of $U\ni \bsy\mapsto G(u(\bsy))$ is
verified in the current setting).
$\ReLU$-NN expression rates then follow %
with Thm.~\ref{thm:DNNGRF}. 
We collect these results in the following proposition.
  \begin{proposition}\label{prop:DNNGRF}
    Let $f\in H^{-1}(\dom)$ and $g(\bsy)=\log(a(\bsy))$ be as in \eqref{eq:GRFg}.  
    Suppose that
      $(\psi_j)_{j\in\N}\subseteq L^\infty(\dom)$ in \eqref{eq:GRFg}
      is such that with $b_j := \| \psi_j \|_{L^\infty(\dom)}$ holds
      $\bsb \in \ell^p(\N)$ for some $0<p<\frac{2}{3}$. Denote the solution of
      \eqref{eq:DiffD} by $u(\bsy)\in H_0^1(\dom)$ whenever
      $g(\bsy)\in L^\infty(\dom)$.

      For a given observable $G\in H^{-1}(\dom)$,
      the map $\bsy\mapsto G(u(\bsy))$ is well-defined as the limit
      \begin{equation*}
       \lim_{N\to\infty}G(u(y_1,\dots,y_N,0,0,\dots))\in L^2(U,\gamma) .
      \end{equation*}
      Moreover,
      for every $\delta>0$ (arbitrarily small) there exists
      $C<\infty$ such that
      for every $N\in\N$
      there exists a $\ReLU$-NN
      $\tilde R_N$
      satisfying
\begin{equation*}
  \norm[L^2(U,\gamma)]{G\circ u - \tilde{R}_N} \le C N^{-\frac{1}{p}+1},
\end{equation*}
and
\begin{equation*}
\size (\tilde{R}_N) \le C N^{1+\delta}, \qquad
\depth(\tilde{R}_N) \le C N^{\delta}.
\end{equation*}
\end{proposition}
\section{Conclusions and extensions}
\label{sec:Concl}
In this paper we discussed the approximation of functions in
$L^2(\R^d,\gamma_d)$ with deep $\ReLU$-neural networks. We proved that
the Hermite polynomials can be approximated at an exponential
convergence rate (in terms of the network size). From this, and
classical bounds on the Hermite coefficients, we deduced that
$\ReLU$-NNs are capable of approximating analytic functions on $\R^d$
that allow holomorphic extensions onto certain strips in the complex
plane at an exponential convergence rate. This result was extended to
the infinite dimensional case $d=\infty$, in which case we showed
algebraic convergence rates for the class of so-called
``$(\bsb,\xi,\delta)$-holomorphic functions''.  This notion has
previously occurred in the literature predominantly for functions with
domain $[-1,1]^\N$. We recently extended this definition to functions
with domain $\R^\N$, and analysed the sparsity properties of this
function class in \cite{DNSZ20}. The present analysis in the case
$d=\infty$ strongly draws from these results. Notably, while the
investigation of the expressivity of $\ReLU$-NNs on function classes
over bounded domains has drawn widespread attention in recent years
(see, e.g., the survey \cite{EPGB21} and the references there), we
provide such results on high-dimensional inputs with unbounded
parameter range.

As an application, we discussed the response map of an elliptic PDE,
whose input is given in the form of a Karhunen-Lo\`eve expansion of a
log-Gaussian random field, and proved that this map can be
approximated at an algebraic convergence rate with $\ReLU$-DNNs. We
emphasize, that similar results will hold also for other well-posed
PDE models with log-GRF input. Moreover, as shown in \cite{DNSZ20},
also Bayesian posterior densities for certain PDE based inverse
problems belong to the class of $(\bsb,\xi,\delta)$-holomorphic
functions. Hence our approximation result may also be applied to such
densities. Therefore our analysis could serve as a starting point for
developing and analysing neural network driven algorithms for
parameter estimation in physical systems.
\appendix
\section{Proof of Thm.~\ref{thm:hille}}\label{app:hille}
We recall some of the main steps of the proof of \cite[Theorem
1]{MR871}, to exhibit the specific bound \eqref{eq:fnbound}, in
particular the claimed dependence of the constants on $\beta$ and $f$.

As in \cite[(3.5)-(3.6)]{MR871}, let $n\in\N$,
\begin{equation*}
  N:=(2n+1)^{1/2}
\end{equation*}
and define for
$z\in S_N:=\set{z\in\C}{\Re[z]\in [-N+1,N-1],~\Im[z]\ge 0}$
\begin{equation}\label{eq:xi}
  \xi(z) :=
  \int_N^z (N^2-t^2)^{1/2}\dd t.
\end{equation}
Due to $\Re[z]\in [-N+1,N-1]$, we have $\Re[N^2-t^2]\ge 0$ for all
complex $t$ in the straight line connecting $N$ and $z$. Throughout
what follows, %
for all $x\in\C$ with $\Re[x]\ge 0$, $x^{1/2}$ is understood as the
complex root with nonnegative real part (cp.~\cite[(3.8)]{MR871}).
Then \eqref{eq:xi} uniquely defines $\xi(z)\in\C$ for all $z\in S_N$.
  
There hold the following properties:
\begin{enumerate}
\item As recalled in \cite[(3.3)-(3.4)]{MR871}, there exist
  holomorphic functions\footnote{These functions are denoted by $h_n$
    in \cite{MR871}.}  $(\tilde h_n)_{n\in\Z}$ such that for all
  $n\in\N$
  \begin{equation}\label{eq:tildehn}
    h_n(z) = \pi^{-3/4}(2^{n} n!)^{1/2}
    \left(\exp\left(\frac{n\pi\ii}{2}\right)\tilde h_{-n-1}(\ii z)+\exp\left(\frac{-n\pi\ii}{2}\right)\tilde h_{-n-1}(-\ii z)\right).
  \end{equation}
\item As argued in \cite[(3.11)]{MR871}, there exists an absolute
  constant $M>0$ such that for all $n\in\N$ and for all
  $z\in S_N$\footnote{This bound holds outside of a neighbourhood of
    the points $\pm N$, which are excluded in our definition of
    $S_N$.} holds
  \begin{equation}\label{eq:thbound}
    |\tilde h_{-n-1}(-\ii z)|\le
    M N^{-n-1} \left|\exp\left( \frac{N^2}{4}\right) \left(1-\frac{z^2}{N^2}\right)^{-1/4}\exp(\ii \xi(z))\right|.
  \end{equation}
\item By \cite[Lemma 2]{MR871}, with the ellipse
  \begin{equation*}
    E(N,\beta):=\setc{x+\ii y\in \C}{\frac{x^2}{N^2}+\frac{y^2}{\beta^2}=1}
  \end{equation*}
  it holds for all $z\in E(N,\beta)\cap S_N$ that
  \begin{equation}\label{eq:int3bound}
    \Im[\xi(z)]+|x|(\beta^2-y^2)^{1/2} \ge \beta N -\frac{5}{24}\frac{\beta^3}{N}.
  \end{equation}
\end{enumerate}

For fixed $n\in\N$, we bound $|\int_\R g(x) h_n(x) \dd x|$.  Since the
integrand is holomorphic in the strip $S_\tau$, the path of
integration may be changed within the strip.  Using \eqref{eq:tildehn}
we can write
\begin{align} \label{eq:C1C2} \int_\R g(x) h_n(x) \dd x &=
  \int_{|x|>N-1} g(x) h_n(x)\dd x+\pi^{-3/4}(2^{n} n!)^{1/2}
  \int_{C_1}
                                                          g(z)\exp\left(\frac{n\pi\ii}{2}\right)\tilde h_{-n-1}(\ii z)\dd z\nonumber\\
                                                        &\quad+\pi^{-3/4}(2^{n}
                                                          n!)^{1/2}
                                                          \int_{C_2}
                                                          g(z)\exp\left(\frac{n\pi\ii}{2}\right)\tilde
                                                          h_{-n-1}(-\ii
                                                          z)\dd z,
\end{align}
where the contours $C_1$ and $C_2$ are sketched in
Fig.~\ref{fig:C1C2}.  %
In the following fix $\beta\in (0,\tau)$.

First, by \eqref{eq:cramer} holds
$\sup_{x\in\R}|h_n(x)|\le \pi^{-1/4}$. Using \eqref{eq:asshille}, we
get
\begin{equation}\label{eq:intlN}
  \left|\int_{|x|>{N-1}} g(x) h_n(x) \dd x\right|
  \le 2 \pi^{-1/4} \int_{N-1}^\infty \exp(-\beta|x|) \dd x
  =\frac{2 \pi^{-1/4}\exp(\beta)}{\beta} \exp(-N\beta).
\end{equation}

Next we bound the integral over $C_2$ in \eqref{eq:C1C2}. By symmetry,
the one over $C_1$ can be treated in the same way. Denote the
intersection of $E(N,\beta)$ with $\set{z\in\C}{\Re[z]=-N+1}$ in the
second quadrant with $P$, and the intersection of $E(N,\beta)$ with
$\set{z\in\C}{\Re[z]=N-1}$ in the first quadrant with $Q$.  Denote the
vertical line connecting $-N+1$ with with $P$ by $v_1$, and the one
connecting $N-1$ with $Q$ by $v_2$.  We start with the integral over
$v_1$ and compute $P$.
We have $\Re[P]=-N+1$. The imaginary part of $P$ is obtained by
solving $\frac{(N-1)^2}{N^2}+\frac{y^2}{\beta^2}=1$ for $y$. This
yields
\begin{equation}\label{eq:ImP}
  \Im[P]=\beta \frac{(2N-1)^{1/2}}{N}.
\end{equation}
We note in passing that \cite{MR871} claims the length of the vertical
parts of the path of integration is $O(n^{-3/4})$, but we obtain
$O(N^{-1/2})=O(n^{-1/4})$.  This shall be, as we show, sufficient to
conclude. By \eqref{eq:asshille} for all $z=-(N-1)+\ii y\in v_1$
\begin{align}\label{eq:int2g}
  |g(z)|&\le B(\beta) \exp\left(-(N-1)(\beta^2-y^2)^{1/2}\right)
          \le B(\beta) \exp\left(-(N-1)\beta \left(1-\frac{2N-1}{N^2}\right)^{1/2}\right)\nonumber\\
        &=B(\beta)\exp\left(-(N-1)\beta \frac{N-1}{N}\right)
          \le B(\beta)\exp(2\beta)\exp(-N\beta).
\end{align}
  
Next observe that for $z=-N+1+\ii y\in v_1$
\begin{align*}
  \Im[\xi(z)] &= \Im\left[\int_{N}^{-N+1}(N^2-t^2)^{1/2}\dd t+\int_{-N+1}^{-N+1+\ii y}(N^2-t^2)^{1/2}\dd t\right]\nonumber\\
              &=\Im\left[\ii\int_0^y(N^2-(N-1+\ii t)^2)^{1/2} \dd t\right].
\end{align*}
The last term is equal to
$\int_0^y\Re[(N^2-(N-1+\ii t)^2)^{1/2}]\dd t$, and this term is
nonnegative by our choice of the branch for the square root and since
$\Re[N^2-(N-1+\ii t)^2]=N^2-(N-1)^2+t^2\ge 0$. Hence
\begin{equation}\label{eq:int2xi}
  |\exp(i\xi(z))|\le 1\qquad\forall z\in v_1.
\end{equation}
  
Next, we bound the term $|1-\frac{z^2}{N^2}|^{-1/4}$ %
occurring in \eqref{eq:thbound}. %
Assume $z=x+\ii y\in S_N$, i.e.\ $x\in [-N+1,N-1]$ and $y\ge 0$.  Then
\begin{align*}
  |N^2-z^2|^2 &= |N^2-(x^2+2\ii x y-y^2)|^2
                =|N^2-x^2+y^2-2\ii x y|^2\nonumber\\
              &=
                (N^2-x^2+y^2)^2+4x^2y^2\ge (N^2-x^2)^2,
\end{align*}
since the minimum is reached for $y=0$.  Hence, using that
$N^2-x^2\ge 2N-1$ if $|x|\le N-1$,
\begin{equation}\label{eq:int2zN}
  \left|1-\frac{z^2}{N^2}\right|^{-1/4}
  =\left|\frac{N^2}{N^2-z^2}\right|^{1/4}
  \le \left(\frac{N^2}{2N-1} \right)^{1/4}
  \le N^{1/4}\qquad\forall z\in S_N
\end{equation}
where we used $N\ge 1$ so that $\frac{N}{2N-1}\le 1$.
Stirling's formula $n!<\e (2\pi n)^{1/2} (\frac{n}{\e})^n$ implies
$(\frac{n}{\e})^{-1/2}n^{-1/2}< (n!)^{-1/2}(2\pi)^{1/4}\e^{1/2}$.
Hence
\begin{equation}\label{eq:int2Nn}
  N^{-n-1}\exp\left(\frac{N^2}{4}\right) = (2n+1)^{-\frac{n+1}{2}}\exp\left(\frac{2n+1}{4}\right) < 2^{-n/2} \left(\frac{n}{\e}\right)^{-1/2}n^{-1/2}<(2\pi\e)^{1/2} (2^nn!)^{-1/2}.
\end{equation}
  
Combining %
\eqref{eq:ImP}-\eqref{eq:int2Nn} with \eqref{eq:thbound} we get
\begin{align}\label{eq:intv1}
  \left|\int_{v_1} g(x)\tilde h_{-n-1}(-\ii z)\dd z\right|
  &\le M B(\beta)\exp(2\beta)\exp(-N\beta)
    \left(\beta \frac{(2N-1)^{1/2}}{N}\right)N^{1/4} \exp\left(\frac{N^2}{4}\right)
    N^{-n-1}\nonumber\\
  &\le 2 M (2\pi\e)^{1/2}B(\beta) \beta\exp(2\beta) (2^n n!)^{-1/2}\exp(-\beta (2n+1)^{1/2}),
\end{align}
where we used $\frac{(2N-1)^{1/2}N^{1/4}}{N}\le 2$ for all $N\ge
1$. The integral over $v_2$ can be treated in the same way.

Finally, denote by $a=E(N,\beta)\cap S_N$ the arc of the ellipse
$E(N,\beta)$ connecting $P$ and $Q$. By \eqref{eq:asshille},
\eqref{eq:thbound}, \eqref{eq:int3bound}, \eqref{eq:int2zN} and
\eqref{eq:int2Nn} we have with $z=x+\ii y$ and because the length of
the arc $a$ is bounded by $2(\beta+N)$
\begin{align}\label{eq:inta}
  \int_a|g(z)\tilde h_{-n-1}(-\ii z)|\dd z
  &\le M B(\beta) N^{-n-1}\exp\left(\frac{N^2}{4}\right) \int_a |\exp(-\ii \xi(z)-|x|(\beta^2-y^2))| \left|1-\frac{z^2}{N^2}\right|^{-1/4}\dd z\nonumber\\
  &\le M (2\pi\e)^{1/2}(2^n n!)^{-1/2} B(\beta) 2(\beta+N) N^{1/4}
    \exp\left(-\beta N+\frac{5}{24}\frac{\beta^3}{N}\right)\nonumber\\
  &\le M(2\pi\e)^{1/2}(2^n n!)^{-1/2}
    B(\beta)  2(1+\beta)\exp\left(\frac{5\beta^3}{24}\right) (1+N)^{5/4}\exp(-\beta N).
\end{align}

Using \eqref{eq:C1C2} and adding up all upper bounds in
\eqref{eq:intlN}, \eqref{eq:intv1} and \eqref{eq:inta} we obtain with
\begin{equation*}
  \tilde K(\beta):= \frac{2\pi^{-1/4}\exp(\beta)}{\beta}+4\left[M\pi^{-3/4}(2\pi\e)^{1/2}\beta\exp(2\beta)\right]
  +2\left[M\pi^{-3/4}(2\pi\e)^{1/2} 2(1+\beta)\exp\left(\frac{5\beta^3}{24}\right)\right]
\end{equation*}
the bound
\begin{equation*}
  \left|\int_\R g(x)h_n(x)\dd x\right|
  \le B(\beta) K(\beta) (1+(2n+1)^{1/2})^{5/4} \exp(-\beta(2n+1)^{1/2}).
\end{equation*}
Since this holds for all $\beta\in (0,\tau)$, absorbing\footnote{Here
  \cite{MR871} obtains a term $O(n^{1/4})$ instead of $O(n^{5/8})$.}
$(1+(2n+1)^{1/2})^{5/4}$ in the exponentially decreasing term, we find
that for all $\beta\in (0,\tau)$ exists $K(\beta)$ depending on
$\beta$ (but not on $n$ or $g$) such that \eqref{eq:fnbound} holds.

\tikzset{ partial ellipse/.style args={#1:#2:#3}{ insert path={+
      (#1:#3) arc (#1:#2:#3)} } }
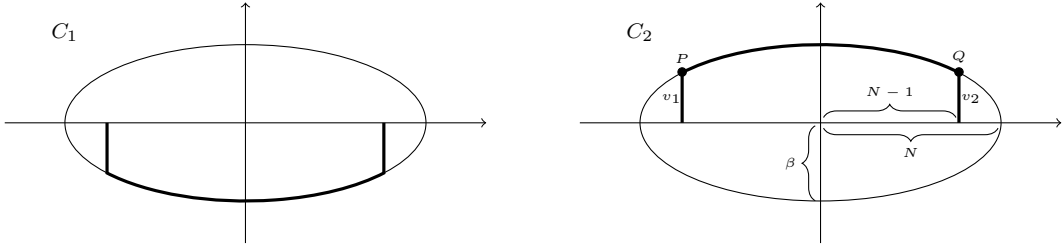
\begin{figure}
  \begin{center}
    \begin{tikzpicture}[scale=0.8]
      \draw[->] (-4,0) -- (4,0); \draw[->] (0,-2) -- (0,2); \draw[-]
      (0,0) ellipse (3cm and 1.3cm); \draw[-,very thick] %
      (-2.3,0) -- (-2.3,-0.845); \draw[-,very thick] %
      (2.3,0) -- (2.3,-0.845); \draw[very thick] (0,0) [partial
      ellipse=180+39.5:180+140.5:3cm and 1.3cm]; \node at (-3,1.5)
      {\footnotesize $C_1$};
    \end{tikzpicture}
    \hspace{1cm}
    \begin{tikzpicture}[scale=0.8]
      \draw[->] (-4,0) -- (4,0); \draw[->] (0,-2) -- (0,2); \draw[-]
      (0,0) ellipse (3cm and 1.3cm); \draw[-,very thick] %
      (-2.3,0) -- (-2.3,0.845); \draw[-,very thick]
      (2.3,0) -- (2.3,0.845); \draw[very thick] (0,0) [partial
      ellipse=39.5:140.5:3cm and 1.3cm];

      \draw
      [decorate,decoration={brace,amplitude=5pt,raise=2pt},yshift=0pt]
      (0.05,0) -- (2.25,0) node [black,midway,yshift=0.4cm] {\tiny
        $N-1$}; \draw
      [decorate,decoration={brace,amplitude=5pt,mirror,raise=2pt},yshift=0pt]
      (0.05,0) -- (2.95,0) node [black,midway,yshift=-0.4cm] {\tiny
        $N$}; \draw
      [decorate,decoration={brace,amplitude=5pt,mirror,raise=2pt},yshift=0pt]
      (0,-0.05) -- (0,-1.295) node [black,midway,xshift=-0.4cm] {\tiny
        $\beta$}; \node at (-3,1.5) {\footnotesize $C_2$}; \filldraw
      (-2.3,0.845) circle (2pt) node [above] {\tiny $P$}; \filldraw
      (2.3,0.845) circle (2pt) node [above] {\tiny $Q$};
      \node at (-2.47,0.4) {\tiny $v_1$}; \node at (2.5,0.4) {\tiny
        $v_2$};
    \end{tikzpicture}
  \end{center}
  \caption{Paths of integration $C_1$ and $C_2$ in \eqref{eq:C1C2}.}
\end{figure}\label{fig:C1C2}

\section{Proof of Thm.~\ref{thm:finite}}
\label{sec:PfThmFinite}
There holds the following Lemma \cite[Lemma
A.2]{MR2318799}\footnote{Lemma A.2 in \cite{MR2318799} is stated only
  for $s\in\N$ and with different constants. The current lemma follows
  by the same argument after adjusting some constants.}:
\begin{lemma}\label{lemma:A7}
  Let $r\in (0,1)$ and $s>0$. Then with $a:=r^{\sqrt{2}}$
  \begin{equation*}
    \sum_{\set{k\in\N_0}{k>s}} r^{\sqrt{2k+1}}            
    \le
    \frac{2}{a(1-a)}(\sqrt{s+2}+4) a^{\sqrt{s}}.
  \end{equation*}
\end{lemma}
    
    \begin{proof}[Proof of Thm.~\ref{thm:finite}]
      With $\Theta$ from Lemma \ref{lemma:iso2}, denote
      $F(\bsx):=\Theta(f)(\bsx)=f(2^{1/2}\bsx)\e^{-\norm[2]{\bsx}^2}\pi^{-d/4}$.
      Then, since $f$ satisfies Assumption \ref{ass:analytic}, for
      every $\bsbeta\in (0,\bstau)$ and every
      $\bsx+\ii\bsy\in S_\bsbeta$ holds
      \begin{align*}
        |F(\bsx+\ii\bsy)|
        &=|f(2^{1/2}\bsx+\ii2^{1/2}\bsy)|\frac{\exp(-\frac{\sum_{j=1}^d x_j^2}{2})}{\pi^{\frac{d}{4}}}\nonumber\\
        &\le B(\bsbeta) \exp\left(\sum_{j=1}^d \left(\frac{(2^{1/2}x_j)^2}{4}-2^{-1/2}|2^{1/2}x_j|(\beta_j^2-\frac{1}{2}(2^{1/2}y_j)^2) \right)\right)\frac{\exp(-\frac{\sum_{j=1}^d x_j^2}{2})}{\pi^{\frac{d}{4}}}\nonumber\\
        &\le \frac{B(\bsbeta)}{\pi^{\frac{d}{4}}}
          \exp\left(\sum_{j=1}^d -|x_j|(\beta_j^2-y_j)^2 \right),
      \end{align*}
      so that $F$ satisfies the assumption of
      Cor.~\ref{cor:hillemulti} with the constant
      $B(\bsbeta)/\pi^{\frac{d}{4}}$. The first item thus follows by
      Corollary \ref{cor:hillemulti} and Lemma \ref{lemma:iso2}.

      To show the second item, we assume in the following
      \eqref{eq:epsconstraint}, which implies by Lemma
      \ref{lemma:epsconstraint} with $\delta(\bsbeta)$ as in
      \eqref{eq:delta}
      \begin{equation}\label{eq:Lepsbound}
        |\Lambda_\eps| 
        \le 2^d
        \frac{\log(\eps)^{2d}}{\prod_{j=1}^d\beta_j^2}
        \;\Rightarrow \;
        |\Lambda_\eps|^{\frac{1}{2d}} 
        \le 2^{\frac{1}{2}}
        \frac{|\log(\eps)| }{\delta(\bsbeta)}.
      \end{equation}
      It suffices to prove the theorem under the constraint
      \eqref{eq:epsconstraint}, since
      $\eps\in (\exp(-\max_{j\in\{1,\dots,d\}} \beta_j),1)$ only
      corresponds to finitely many sets $\Lambda_\eps$.
      
      Denote $\chi_j:=(\log(\eps))^2/\beta_j^2$.  Then with
      $K(\bsbeta):=\prod_{j=1}^dK(\beta_j)$, using that
      $\bsnu\in\Lambda_\eps$ iff $\nu_j\le \chi_j$ for all $j$
      (cp.~\eqref{eq:Leps}),
      \begin{align*}
        &\normc[L^2(\R^d,\gamma_d)]{f-\sum_{\bsnu\in\Lambda_\eps}\dup{f}{H_\bsnu}H_\bsnu}^2
          = \sum_{\bsnu\in\N_0^d\backslash\Lambda_\eps} |\dup{f}{H_\bsnu}|^2\nonumber\\
        &\qquad\qquad\le
          K(\bsbeta)^2
          B(\bsbeta)^2\sum_{j=1}^d\sum_{n>\chi_j}\exp(-2\beta_j\sqrt{2n+1})
          \sum_{\set{(\nu_i)_{i\neq j}}{\nu_i\in\N_0}}\prod_{i\neq j} \exp(-2\beta_i\sqrt{2\nu_i+1}).
      \end{align*}
      By Lemma \ref{lemma:A7}, with $a_j:=\exp(-2^{3/2}\beta_j)$ and
      $C_j:=\frac{2}{a_j(1-a_j)}$,
      \begin{equation*}
        \sum_{n>\chi_j}\exp(-2\beta_j\sqrt{2n+1})
        \le C_j (\sqrt{\chi_j+2}+4) \exp(-2^{3/2}\beta_j\sqrt{\chi_j}) 
        \le \tilde C_j \exp(-2\beta_j\sqrt{\chi_j})=\tilde C_j\eps^2,
      \end{equation*}
      for some $\tilde C_j$ depending on $\beta_j$.  Furthermore
      (e.g.\ by Lemma \ref{lemma:A7}) we have
      $\sum_{n\in\N_0}\exp(-2\beta_j\sqrt{2n+1})=:D(\beta_j)<\infty$,
      and thus
      \begin{equation*}
        \sum_{\set{(\nu_i)_{i\neq j}}{\nu_i\in\N_0}}\prod_{i\neq j} \exp(-2\beta_i\sqrt{2\nu_i+1})
        =\prod_{i\neq j} \sum_{n\in\N_0}\exp(-2\beta_i\sqrt{2n+1})
        \le (\max_{j\le d}D(\beta_j))^{d-1}=:C_{\rm max}.
      \end{equation*}
      Hence
      \begin{equation}\label{eq:normdiffeps}
        \normc[L^2(\R^d,\gamma_d)]{f-\sum_{\bsnu\in\Lambda_\eps}\dup{f}{H_\bsnu}H_\bsnu}^2
        \le 
        K(\bsbeta)^2B(\bsbeta)^2 C_{\rm max} \sum_{j=1}^d \tilde C_j \eps^2 = C\eps^2,
      \end{equation}
      with the $\bsbeta$ and $d$-dependent constant
      $C:=K(\bsbeta)^2 B(\bsbeta)^2 C_{\rm max}\sum_{j=1}^d \tilde
      C_j$.
      
      By \eqref{eq:Lepsbound} %
      \begin{equation*}
        \eps
        \le 
        \exp\left(-2^{-\frac{1}{2}}
          \delta(\bsbeta) |\Lambda_\eps|^{\frac{1}{2d}} \right) 
      \end{equation*}
      so that together with \eqref{eq:normdiffeps}
      \begin{equation*}
        \normc[L^2(\R^d,\gamma_d)]{f-\sum_{\bsnu\in\Lambda_\eps}\dup{f}{H_\bsnu}H_\bsnu}
        \le C 
        \exp(-2^{-\frac 1 2}\delta(\bsbeta)|\Lambda_\eps|^{\frac{1}{2d}}).\qedhere
      \end{equation*}
    \end{proof}
    \bibliographystyle{abbrv} \bibliography{main}
\end{document}